\documentclass[11pt]{article}
\usepackage{subfigure}
\usepackage{amssymb,amsmath,amsthm,epsfig,a4,verbatim,pstricks,ifthen,setspace,ulem} 
\usepackage{url}
\usepackage{float}
\allowdisplaybreaks
\definecolor{mygreen}{rgb}{0.05,0.6,0.05}
\newtheorem{thm}{\sc Theorem.}[section]
\newtheorem{lem}{\sc Lemma.}[section]

\newtheorem{rem}{\sc Remark.}[section]
\renewcommand{\theequation}{\arabic{section}.\arabic{equation}}
\newenvironment{AMS}%
{{\upshape\bfseries AMS subject classifications. }\ignorespaces}{}
\newenvironment{keywords}{{\upshape\bfseries Key words. }\ignorespaces}{}

\newcommand{\Vh}{V_h}

\newcommand{\ds}{\; {\rm d}s}

\newcommand{\dt}{\; {\rm d}t}
\newcommand{\dS}{\,{\rm d}S}

\newcommand{\ah}{a_h}

\newcommand{\eht}{\tilde{e}_h}

\newcommand{\vr}{\varrho}

\newcommand{\ut}{u^e}
\newcommand{\ft}{f^e}
\def\ba{\begin{align}}
\def\ea{\end{align}}
\newcommand{\nn}{\nonumber}
\newcommand{\eps}{\varepsilon}
\newcommand{\Th}{\mathcal{T}_h}

\newcommand{\Trt}{\widetilde{\mathcal{T}}_h}

\newcommand{\Vr}{\widetilde{V}_h}

\newcommand{\Dr}{D^{\vr}_h}

\def\epsilon{\varepsilon} 
\def\dx{\,{\rm d}x}
\def\hat{\widehat}

\begin{document}
\title{A Practical Phase Field Method for\\ an Elliptic Surface PDE}
\author{John W. Barrett\footnotemark[2] \and 
        Klaus Deckelnick  \footnotemark[3]\ \and 
        Vanessa Styles \footnotemark[4]}

\renewcommand{\thefootnote}{\fnsymbol{footnote}}
\footnotetext[2]{John passed away on 30 June 2019. We dedicate this article to his memory.}
\footnotetext[3] {Institut f{\"u}r Analysis und Numerik,
 Otto-von-Guericke-Universit{\"a}t Magdeburg, 39106 
Magdeburg, Germany}
\footnotetext[4]{Department of Mathematics, University of Sussex, Brighton, BN1 9RF, UK}

\date{}

\maketitle

\begin{abstract}
We consider a diffuse interface approach for solving an elliptic PDE on a given closed hypersurface.  The method is based on
a (bulk) finite element scheme employing numerical quadrature for  the phase field function and hence is very easy to implement compared to other approaches.  We estimate
the error in natural norms in terms of the spatial grid size, the interface width and the order of the underlying quadrature rule. Numerical test calculations are presented which
confirm the form of the error bounds.
\end{abstract} 

\begin{keywords} 
elliptic surface PDE, diffuse interface, finite element method, error analysis
\end{keywords}

\begin{AMS} 35R01, 65M60, 65M15 
\end{AMS}
\renewcommand{\thefootnote}{\arabic{footnote}}

\section{Introduction} \label{sec:1}

Let $\Gamma \subset \mathbb{R}^{n+1} \, (n=1,2)$ be a closed hypersurface. In this paper we are
concerned with a phase field approach for the numerical solution of the PDE
\ba
-\Delta_{\Gamma} u + u &= f \qquad \mbox{on } \Gamma
\label{P}
\end{align} 
and more general elliptic PDEs on surfaces.
Here, $\Delta_{\Gamma}$ denotes the Laplace--Beltrami operator and $f$ is a given function
on $\Gamma$. Apart from being of interest in their own right, 
elliptic surface PDEs may arise as subproblems in
the time discretization of parabolic surface PDEs as well as in systems involving a coupling to a bulk PDE (see e.g. \cite{ER13}). \\
A major issue in the design and analysis of numerical methods for (\ref{P}) lies
in the fact that the simultaneous approximation of the PDE and of the surface $\Gamma$ is
required. Let us briefly review the various computational approaches that have been 
suggested in the literature. Further references can be found in the nice review articles \cite{DzE13}  and \cite{BDN}. \\
In his seminal paper \cite{Dz88}, Dziuk proposes and analyzes
a method that employs continuous, piecewise linear finite elements on a regular simplicial partitioning
of $\Gamma_h$, a polyhedral approximation of $\Gamma$.   
This approach has been extended to higher order FEM spaces and higher order
polynomial approximations of $\Gamma$ by Demlow in \cite{De09}, while an adaptive version 
of the method can be found  
in \cite{DD07}. However, the construction of a regular
polynomial approximation may be difficult in practice, in particular if the surface is given 
implicitly in terms of a  level set  function. The trace finite element method, proposed by
Olshanskii, Reusken and Grande in \cite{ORG09}, is based on a background mesh which induces an unfitted
approximation $\Gamma_h$ of $\Gamma$ and employs traces of bulk finite element functions. 
Even though $\Gamma_h$ is in general not regular, optimal error estimates 
for piecewise linear
finite elements are obtained. Further developments and variants of this trace method (also called cut finite element method)
can be found in \cite{OR10}, \cite[Section 3]{DER14}, \cite{R15}, \cite{GLR}, \cite{DO12}, \cite{BHL15} and \cite{BHLMZ}. \\  
In the case of a level set representation of $\Gamma$ there is a class of methods that is based on
extending the PDE ({\ref{P}) to an
open neighborhood of $\Gamma$. Using earlier ideas of \cite{BCOS}, 
Burger  considers in \cite[Section 2]{Bu09} an extension with the property that (\ref{P}) is satisfied
simultaneously on all neighboring level surfaces. This approach  gives rise to a weakly
elliptic bulk PDE,  which is  degenerate in the direction normal to the level surfaces 
and which can be solved numerically with the help
of standard bulk
finite elements. Error estimates have been derived in \cite[Theorem 6]{Bu09}, while
\cite{DDEH} considers the problem in a narrow band of width $h$ around $\Gamma$ and 
provides an $O(h)$ bound in $H^1(\Gamma)$.
In both cases the corresponding error analysis  is complicated
by the degeneracy of the extended PDE; an extended PDE, which is uniformly elliptic, has been proposed in  
\cite{CO13} and \cite{OS16} and involves the mean curvature of $\Gamma$.  A different method which leads to a
uniformly elliptic bulk PDE, is obtained by considering the equation which is satisfied by a natural extension
of the solution of the surface PDE. If $\Gamma$ is given implicitly in terms of the signed distance function
this extension is the function which is constant in normal direction, and one is led to the closest point method,
see \cite{MR} for the parabolic case. In the case of a general level set function the corresponding PDE has
been derived in \cite{DER14}, where unfitted sharp and narrow band  finite element methods have been proposed and
analyzed. \\
Note that for schemes that are based on an implicit representation of the types described 
above the discrete surface or the boundary of a narrow
band may cut arbitrarly through a bulk element. Locating these cuts and integrating over 
the discrete surface or partial elements is in general cumbersome. 
A way to circumvent these difficulties is offered by the use of a diffuse interface method. The 
starting point of this approach is again an extension of the surface PDE 
to a neighborhood of $\Gamma$, which is then localized to a thin layer of width proportional to 
$\epsilon$ 
with the help of a phase field function. The resulting problem can be
solved using finite elements, where the geometry is now resolved by evaluating the phase field 
function. This approach was suggested and analyzed in \cite[Section 3]{Bu09} in the elliptic case, and in
\cite{RV} for a linear diffusion equation
for a phase field function with nonlocal support. In \cite{ES09}, \cite{ESSW11} and \cite{DS} a phase field
function with compact support was used in the approximation of an advection diffusion equation on a
moving surface.
In practice, numerical integration needs
to be used which now becomes an issue as estimates for the resulting error  require derivatives 
of the phase field function, which scale with $\epsilon^{-1}$.  Our main contribution in this paper is 
a new, fully practical phase field method to solve (\ref{P})  together with a corresponding error analysis in
natural norms. Furthermore we shall present test calculations for hypersurfaces in two and three dimensions
which confirm the form of our error bounds.

\section{Preliminaries} \label{sec:2}
\setcounter{equation}{0}

\subsection{Notation and problem formulation}
Let $\Gamma \subset \mathbb{R}^{n+1} \, (n=1,2)$ be a  smooth, connected, compact and 
orientable hypersurface without boundary. In view of the Jordan-Brouwer separation theorem,
$\Gamma$ divides $\mathbb{R}^{n+1}$ into an interior and an exterior domain and we denote by $d$ the
signed distance function to $\Gamma$ oriented in such a way that $d<0$ in the interior, $d>0$ in the
exterior of $\Gamma$. It is well--known (see \cite{GT}, Section 14.6) that there exists an open 
neighbourhood $\Omega$ of $\Gamma$ such that $d$ is smooth in $\Omega$ with  $| \nabla d(x) |=1, x \in \Omega$ as well as $\nabla d(x)= \nu(x), x \in \Gamma$, where
$\nu(x)$ is the unit outer normal to $\Gamma$. 
Furthermore, the function $\hat p(x):= x - d(x) \, \nabla d(x)$ assigns to every $x \in \Omega$ the closest point on $\Gamma$, so that
\begin{equation}  \label{op}
\hat{p}(x) \in \Gamma, \qquad  x-\hat{p}(x) \perp T_{\hat{p}(x)} \Gamma \qquad \forall \ x \in \Omega,
\end{equation}
where  $T_p \Gamma$ denotes the tangent space at $p \in \Gamma$.
Note that $\nabla d(x) = \nabla d(\hat p(x)), x \in \Omega$.
For a differentiable function $\eta: \Gamma \rightarrow \mathbb{R}$ let
$\nabla_{\Gamma} \eta(x) = (\underline{D}_1 \eta(x),\ldots,\underline{D}_{n+1} \eta(x)) 
\in T_x \Gamma$ be its tangential gradient. We have that
\begin{equation} \label{tanggrad}
\nabla_{\Gamma} \eta(x)= \nabla \bar{\eta}(x) - (\nabla \bar{\eta}(x) \cdot \nu(x))\, \nu(x) 
\qquad \forall \ x \in \Gamma,
\end{equation}
where $\bar{\eta}$ is an extension of $\eta$ to an open neighborhood of $\Gamma$.  \\[2mm]
Let us consider the following elliptic PDE in divergence form
\begin{equation}  \label{GP}
- \sum_{i,j=1}^{n+1} \underline{D}_i \left( a_{ij} \,\underline{D}_j u \right) + a_0 \,u =f 
\qquad \mbox{ on } \Gamma.
\end{equation}
We assume that $a_{ij} \in C^2(\Gamma)$ and  that 
$A(x)=(a_{ij}(x))_{i,j=1}^{n+1}$ defines a symmetric, uniformly positive definite linear
map from $T_x \Gamma$ into itself, so that there exists $\alpha>0$ with
\begin{subequations}
\begin{alignat}{2} 
\sum_{i,j=1}^{n+1} a_{ij}(x)\, \xi_i \, \nu_j(x) = 0 \qquad &&  \forall \  \xi \in T_x \Gamma, 
\ x \in \Gamma;  \label{comp} \\
\sum_{i,j=1}^{n+1} a_{ij}(x) \,\xi_i \,\xi_j \geq \alpha \,| \xi |^2 \qquad && \forall \ \xi \in 
T_x \Gamma, \ x \in \Gamma. \label{ellipt}
\end{alignat}
\end{subequations}
Since $A(x)\, \nu(x)$ is not relevant for (\ref{GP}), we may assume that $A(x) \,\nu(x) = \nu(x)$, so that
\begin{equation} \label{ell1}
\sum_{i,j=1}^{n+1} a_{ij}(x) \,\xi_i\, \xi_j \geq \min\{\alpha,1\} \, | \xi |^2 \qquad  \forall  
\ \xi \in \mathbb{R}^{n+1}, \ x \in \Gamma. 
\end{equation}
Furthermore, we suppose that $a_0$ and $f$ belong to $W^{1,\infty}(\Gamma)$  and that there exists 
$\alpha_0>0$ such that
\begin{equation}  \label{a0}
a_0(x) \geq \alpha_0 \qquad \forall \ x \in \Gamma.
\end{equation}
It follows from the Lax--Milgram lemma that for every $f \in L^2(\Gamma)$ 
the PDE  (\ref{GP}) has a unique weak solution $u \in H^1(\Gamma)$ in the sense that
\begin{equation}
\sum_{i,j=1}^{n+1} \int_{\Gamma} a_{ij}\, \underline{D}_j u \, \underline{D}_i v \dS_p + 
\int_{\Gamma} 
a_0\, u \, v \dS_p= \int_{\Gamma}  f \, v \dS_p\qquad \forall\ v \in H^1(\Gamma),
\label{WFG}
\end{equation}
where $\dS_p$ is the surface element of $\Gamma$. 
Furthermore, standard regularity theory implies that $u \in H^2(\Gamma)$ and
\begin{equation} \label{apriori}
\Vert u \Vert_{H^2(\Gamma)} \leq C\, \Vert f \Vert_{L^2(\Gamma)}.
\end{equation}
In what follows we suppose that $\Gamma$ is represented in implicit form, i.e. there 
exists a smooth function $\phi: \bar \Omega \rightarrow \mathbb{R}$ such that
\begin{equation} \label{levset}
\Gamma = \left \{ x \in \Omega : \phi(x) = 0 \right \} \qquad \mbox{and} \qquad 
\nabla \phi(x) \neq 0 \quad \forall \ x \in \Gamma.
\end{equation}
By choosing $\Omega$ smaller if necessary  we may assume the
existence of $c_1 \geq c_0 >0 $ such that
\begin{equation}  \label{gradbounds}
\displaystyle 
c_0 \leq | \nabla \phi(x) | \leq c_1 \quad \forall \ x \in \overline{\Omega}.
\end{equation}

\subsection{Extension} \label{exten}
As already mentioned in the introduction our numerical approach is based on extending surface 
quantities and  the surface PDE to a neighborhood of $\Gamma$. In what follows we abbreviate 
\begin{displaymath}
U_r:= \left \{ x \in \Omega :  | \phi(x) | < r \right \}.
\end{displaymath}
A common way to extend a given function $g: \Gamma \rightarrow \mathbb{R}$ consists in setting 
$\hat{g}(x):= g(\hat{p}(x))$, often called the closest-point extension, and  we  shall use $\hat{p}$ in order to extend the data $a_{ij},a_0$ and $f$ to
a neighbourhood of $\Gamma$. However, in order to derive our scheme and in order to carry out the
error analysis we require a further extension which is better adapted to the level set function $\phi$ and
the diffusion matrix $A$, see in particular the relation (\ref{orth}) below.
In what follows we generalize  ideas from  
\cite[Section 2.1]{DER14}. Consider
for $p \in \Gamma$ the parameter-dependent system of ODEs
\begin{equation}  \label{odesys}
\gamma_p'(s) = \frac{A(p)\,\nabla \phi(\gamma_p(s))}{A(p) \,\nabla \phi(\gamma_p(s)) 
\cdot \nabla \phi(\gamma_p(s))}, \qquad \gamma_p(0)=p.
\end{equation}
It is not difficult to see that there is $\delta>0$ such that the solution 
$\gamma_p$ of (\ref{odesys}) exists
uniquely on $(-\delta,\delta)$ for every $p \in \Gamma$, so that we may define the mapping
$F:\Gamma \times (-\delta,\delta)
\rightarrow \mathbb{R}^{n+1}$ by $F(p,s):=\gamma_p(s)$. Recalling that $a_{ij} \in C^2(\Gamma)$ we infer with the help of well--known results on the differentiability of solutions of ODEs with respect to
parameters and initial conditions that $F \in C^2(\Gamma \times (-\delta,\delta);\mathbb{R}^{n+1})$. Furthermore,  (\ref{odesys}) implies that
\begin{displaymath}
\frac{\rm d}{{\rm d}s} \phi(\gamma_p(s)) = \nabla \phi(\gamma_p(s)) \cdot \gamma_p'(s) = 1,
\end{displaymath}
which yields $\phi(\gamma_p(s))=s$ for $|s| < \delta, \ p \in \Gamma$, 
since $\phi(\gamma_p(0))=\phi(p)=0$.  
In particular, $F$ is a bijection from  $\Gamma \times (-\delta,\delta)$ onto $U_{\delta}$ with inverse
\begin{align}
F^{-1}(x)&=(p(x),\phi(x))\qquad \forall \ x \in U_{\delta},
\label{Finv}
\end{align}
where  $p \in C^2(U_{\delta};\mathbb{R}^{n+1})$ satisfies 
\begin{equation}  \label{pprop}
p(x) \in \Gamma \mbox{ for } x \in U_{\delta} \quad \mbox{ and } \quad p(x)=x \mbox{ for }x \in \Gamma.
\end{equation}
It is not difficult to verify that $p(x)=\hat{p}(x)$ in the case $A=I$ and $\phi=d$. \\
Using $p$ we may
define an alternative extension for a given 
$u:\Gamma \rightarrow \mathbb{R}$ to $U_{\delta}$ by setting
\begin{equation} \label{extension}
u^e(x):= u(p(x)) \qquad \forall \ x \in U_{\delta}.
\end{equation}
It is easily seen that $p(\gamma_p(s))=p$, $p \in \Gamma$, so 
that $s \mapsto u^e(\gamma_p(s))$ is constant on $(-\delta,\delta)$. Differentiation with
respect to $s$, together with (\ref{odesys}), then implies that
\begin{equation} \label{orth}
\nabla u^e(x) \cdot A(p(x)) \,\nabla \phi(x) =0 \qquad \forall \ x \in U_{\delta}.
\end{equation}
Suppose in addition that $u$ is a solution of the surface PDE (\ref{GP}). It is shown  in Lemma \ref{extde} of the Appendix that 
$u^e$ then satisfies  the uniformly elliptic PDE
\begin{equation} \label{EP}
\displaystyle 
- \frac{1}{| \nabla \phi | } \,\nabla \cdot \left( A^e \, \nabla \ut \,  | \nabla \phi |  
 \right)  + a^e_0 \, \ut = \ft + \phi \,R \qquad \mbox{in } U_{\delta}, 
\end{equation} 
with $ A^e(x) := A(p(x)), \, a^e_0(x):= a_0(p(x)), \,  f^e(x):= f(p(x))$ and
\begin{equation}  \label{rest}
R(x)  = \sum_{1 \leq | \kappa | \leq 2} \bigl( b_\kappa(x) + \phi(x) c_\kappa(x) \bigr) D^{\kappa}_{\Gamma} u(p(x)), 
\end{equation}
where $b_\kappa \in C^1(U_{\delta}), c_\kappa \in C^0(U_{\delta})$.

\subsection{Phase field approach and finite element approximation}

Let us next derive a suitable localized weak formulation of (\ref{EP}),  
which we shall use later in order to formulate our numerical scheme. Let $\sigma \in C^0(\mathbb{R})$ be such
that $\sigma(r)>0, |r| < \bar r$ and $\sigma(r)=0, |r| \geq \bar r $. A concrete choice of $\sigma$ will be made later.
For  $\eps \in (0,\textstyle\frac{\delta}{\bar r})$  we define 
the phase field function
\ba
\vr(x) := \sigma\left( \frac{\phi(x)}{\eps}\right)
\qquad \forall \ x \in \Omega.
\label{vr}
\end{align}
The restriction on $\eps$ ensures that supp$(\vr) = \overline{U_{\eps \bar r}} \subset U_\delta$.
For a function $v \in L^1(\Omega)$ we obtain with the help of the coarea formula
\begin{equation}  \label{coarea}
\int_{\Omega} v(x) \,  \vr(x) \, | \nabla \phi(x) | \dx = 
\int_{-\eps \bar r }^{\eps \bar r } 
\sigma\left(\frac{t}{\eps}\right) \int_{\lbrace \phi=t \rbrace} v(x) \dS \dt \approx 
\hat{c}\, \,\eps \int_{\lbrace \phi=0 \rbrace} v(x) \dS_p
\end{equation}
for small $\eps>0$, where $\hat{c} = \int_{-\bar r}^{\bar r} \sigma(s) \ds$. 
It is therefore reasonable to approximate the 
surface integral $\int_{\Gamma} v(x) \dS_p$
by the volume integral $(\hat{c}\,\, \eps)^{-1} \int_{\Omega} v(x) \,  
\vr(x) \, | \nabla \phi(x) | \dx$. The latter expression explains the scaling factor $\eps^{-1}$ and
the weight $\vr \, | \nabla \phi |$, which will frequently occur. \\[2mm]
Let us now multiply  (\ref{EP}) by $v\,  \vr \, | \nabla \phi |$ with $v \in H^1(U_r)$ for some $0<r<\delta$
and integrate over $U_r$.  For the leading term
we obtain with the help of integration by parts
\begin{displaymath}
- \int_{U_r} \nabla \cdot \left( A^e \, \nabla \ut \,  | \nabla \phi |   \right) v \, \vr  \dx 
= \int_{U_r} A^e \, \nabla \ut \cdot \nabla v \, \vr \, | \nabla \phi | \dx,
\end{displaymath}
where we have used (\ref{orth}) to see that $A^e \, \nabla \ut \cdot \nabla \vr = 0$. For the same reason the boundary term vanishes as the
unit outer normal to $\partial U_r$ is a multiple of $\nabla \phi$. Thus, we obtain that
\begin{equation} \label{WF}
\int_{U_r}  [ A^e \, \nabla \ut \cdot \nabla v + a^e_0\, \ut \, v ] \,  \vr \,  | \nabla \phi | \dx = \int_{U_r} 
[ \ft + \phi \, R] \, v \, \vr \, | \nabla \phi | \dx \quad \forall v \in H^1(U_r).
\end{equation}
We now use this relation in order to introduce our numerical scheme. To do so, let us assume for simplicity that $\Omega$ is polyhedral and denote by
$\Th$  a regular partitioning of $\Omega$ into simplices 
$T$, i.e.
\ba
\overline{\Omega} = \bigcup_{T \in \Th} T.
\label{Th}
\end{align}
We set $h_T := {\rm diam}(T)$,  $h:=\max_{T \in \Th} h_T$ and  let 
\ba
\Vh := \left\{ \phi_h \in C(\overline{\Omega}) : \phi_h \mbox{ is affine on } T \mbox{ for all } T \in \Th \right\} 
\subset H^1(\Omega).
\label{Vh}
\end{align}
We denote by $I_h : C(\overline{\Omega})
\rightarrow \Vh$ the Lagrange interpolation operator. Note for 
$q > \frac{n+1}{m}$, $m=1$ or $2$ and $\ell=0$ or $1$ that
\ba
|(I-I_h)v|_{\ell,q,T} &\leq C\,h_{T}^{m-\ell}\,|v|_{m,q,T}
\qquad \forall \ v \in W^{m,q}(T), 
\qquad \forall \ T \in \Th.
\label{inter}
\end{align}
In particular we infer from (\ref{gradbounds})  
that there exists an $h_0 >0$
such that for all $h \in (0,h_0]$ 
\begin{align}
\frac{c_0}{2} \leq |\nabla I_h \phi(x)| \leq  2\,c_1\qquad \forall \ x \in \overline{\Omega}. 
\label{Ihgradbelow}
\end{align}
Next, let $\hat{T}$ be the unit simplex in $\mathbb{R}^{n+1}$ and 
\begin{displaymath}
\hat{Q}(\hat g)= | \hat{T} \,  | \sum_{i=1}^L \omega_i \,  \hat g(\hat{b_i}), \qquad \omega_i>0, \, \hat{b}_i \in \hat T, \, i=1,\ldots,L
\end{displaymath}
a quadrature rule which is exact for all polynomials of degree $\leq q$. This gives rise to a quadrature rule on $T$ via
\begin{displaymath}
Q_T(g)= |T| \,  \sum_{i=1}^L \omega_i \,  g(b_{i,T}),
\end{displaymath}
where $b_{i,T}= \Phi_T(\hat{b_i}) \in T$ and $\Phi_T$ is the usual affine transformation from $\hat{T}$ onto $T$. Using a
standard application of the Bramble--Hilbert lemma we obtain for the quadrature error $E_T(g):= Q_T(g) - \int_T g(x) \dx$ that
\begin{equation}  \label{qerror}
| E_T(g) | \leq C \, | T| \,  h_T^{q+1} | g |_{q+1,\infty,T}, \qquad g \in W^{q+1,\infty}(T).
\end{equation}
The degree of exactness of the quadrature formula now enters our choice of 
profile function $\sigma$, which we define as
\ba
\sigma(r) := \left\{\begin{array}{ll}
\cos^{2(q+1)}(r) \qquad & |r| \leq \frac{\pi}{2}, \\
0 & |r| > \frac{\pi}{2}.
\end{array}
\right.
\label{sigma}
\end{align}
A straightforward calculation shows that the corresponding phase field function $\vr(x)=\sigma(\frac{\phi(x)}{\eps})$ satisfies for  $0<| \alpha| \leq q+1$ and $x \in \overline{U_{\epsilon \frac{\pi}{2}}}$
\begin{equation}  \label{vrd}
| D^\alpha \vr(x)|  \leq    C \sum_{k=1}^{| \alpha |} \eps^{-k} | \sigma^{(k)} \bigl( \frac{\phi(x)}{\eps} \bigr)  |
\leq C \sum_{k=1}^{| \alpha |} \eps^{-k} \cos^{2(q+1)-k} \bigl( \frac{\phi(x)}{\eps} \bigr) \leq   C \eps^{-| \alpha |} \vr(x)^{\frac{2(q+1)- | \alpha|}{2(q+1)}}. 
\end{equation}
In order to set up our numerical scheme we define for $h<\eps$ 
\begin{equation}
\Trt:= \left \{ T \in \Th :  | \phi(b_{i,T}) | \leq \eps \arccos \bigl(\frac{h}{\eps} \bigr) \mbox{ for all } i=1,\ldots,L  \right\},
\label{Tht}
\end{equation}
giving rise to the computational domain 
\begin{displaymath}
D_h  := \bigcup_{T \in \Trt} T.
\end{displaymath}

\begin{lem} Denote by $r_0 \in (0,1)$ the unique zero of the function $r \mapsto \arccos(r) - c_1 r$ and set $\hat \eps = \epsilon \arccos(\frac{h}{\eps}) - c_1 h, c_2 = \frac{\pi}{2} + c_1$ ($c_1$ as in (\ref{gradbounds})).  
Suppose that $\eps =\gamma h$ for some $\gamma > \frac{1}{r_0}$ and that $c_2 \epsilon < \delta$. Then we have $\hat \epsilon>0$ and
\begin{equation} \label{inclusions}
\Gamma \subset U_{\hat \eps}  \subset  D_h \subset U_{c_2 \eps} \subset U_{\delta}.
\end{equation}
 Furthermore,
\begin{equation} \label{rholow}
\vr(x)   \leq  C \bigl( \frac{h}{\eps} \bigr)^{2(q+1)}, \qquad x \in D_h \setminus U_{\hat \eps}. 
\end{equation}
\end{lem}
\begin{proof} Since $\frac{h}{\eps}=\frac{1}{\gamma} < r_0$ and $r \mapsto \arccos(r) - c_1 r$ is strictly decreasing we have
\begin{displaymath} 
 \hat \eps = \eps \left( \arccos \bigl(\frac{h}{\eps} \bigr) - c_1 \frac{h}{\eps} \right) > \eps \left( \arccos(r_0) - c_1 r_0 \right)=0.
\end{displaymath}
In particular, $\Gamma= \lbrace x \in \Omega \, : \, \phi(x) =0 \rbrace \subset U_{\hat \eps}$. Next, if $x \in U_{\hat \eps} \cap T$, then 
\begin{displaymath}
| \phi(b_{i,T}) | \leq | \phi(x) | + | \phi(b_{i,T}) - \phi(x) | < \hat \eps + c_1 h = \eps \arccos \bigl(\frac{h}{\eps} \bigr), \quad i=1,\ldots,L,
\end{displaymath}
which implies that $x \in D_h$. Similarly, we see that $D_h \subset U_{\tilde \eps}$, where $\tilde \eps= \eps \arccos \bigl(\frac{h}{\eps} \bigr) + c_1 h \leq c_2 \eps$.
It remains to show (\ref{rholow}) for  $x \in D_h \setminus U_{\hat \eps}$. We may assume in addition that $x \in \overline{U_{\epsilon \frac{\pi}{2}}}$ as otherwise $\vr(x)=0$.
Then we have
\begin{displaymath}
\cos \bigl( \frac{\phi(x)}{\eps} \bigr) \leq \cos \bigl( \frac{\hat \eps}{\eps} \bigr) = \cos \bigl( \arccos (\frac{h}{\eps}) - c_1 \frac{h}{\eps} \bigr) \leq C \frac{h}{\eps},
\end{displaymath}
which yields the desired estimate.
\end{proof}

Next, let us define  the finite element space
\begin{equation} \label{Vr}
\Vr := \left\{ \phi_h \in C(D_h) : \phi_h \mbox{ is affine on } T
\mbox{ for all }  T \in \Trt \right\}.
\end{equation}
Motivated by (\ref{WF}), our fully practical  scheme reads: Find $u_h \in \Vr$ such that
\begin{equation} \label{WFh}
\ah(u_h,v_h) = l_h(v_h)  \qquad \forall \ v_h \in \Vr,
\end{equation}
where the   forms $a_h$ and $l_h$ are given by
\begin{subequations}
\begin{align}
\ah(v_1,v_2) & :=  
\eps^{-1} \sum_{T \in \Trt} Q_T \left[  \vr \, I_h \hat{A} \,\nabla v_1 \cdot \nabla v_2  
+ \vr \, I_h \hat{a}_0\,v_1\,v_2 \right] \, 
| \nabla I_h \phi_{|T} |, 
\label{ah}\\
l_h(v) &:= 
\eps^{-1} \sum_{T \in \Trt} Q_T \left[ \vr\, I_h \hat{f} \,v \right] \, | \nabla I_h \phi_{|T}  |.  \label{defgh}
\end{align}
\end{subequations}
Furthermore, we have abbreviated $\hat{A}(x):= A(\hat{p}(x)), \,  \hat{a}_0(x):= a_0(\hat{p}(x)), \hat{f}(x):= f(\hat{p}(x))$
and remark that these are used in the scheme  rather than
 $A^e$, $a_0^e$ and $f^e$, since in practice the evaluation of $\hat{p}$
is easier compared to $p$.

\begin{rem} \label{practical}
In contrast to other methods, which require the determination of and integration over an 
approximate surface
$\Gamma_h$ or a suitable narrow band, 
the implementation of (\ref{WFh}) is rather straightforward. The underlying
geometry is incorporated through the level set function $\phi$ and the projection $\hat{p}$. Note that $\hat p$ is only
required at the grid points of $\Trt$.  
\end{rem}
Let us introduce
\begin{equation}  \label{normh}
\displaystyle
\Vert v_h \Vert_h := \left( \eps^{-1} \sum_{T \in \Trt} Q_T \left[ \vr \bigl( | v_h |^2 + | \nabla v_h |^2  \bigr) \right] \right)^{\frac{1}{2}}, \quad v_h \in \Vr.
\end{equation}
In view of (\ref{Ihgradbelow}), (\ref{ell1}) and (\ref{a0}) there exists $c_3>0$, which is independent of $h$, such that
\begin{equation} \label{normeq1}
c_3 \Vert v_h \Vert_h^2 \leq a_h(v_h,v_h) \qquad \mbox{ for all } v_h \in \Vr.
\end{equation}
In particular we have:
\begin{lem} The discrete problem {\rm (\ref{WFh})} has a unique solution $u_h \in \Vr$ 
for all $0< h <\eps$, $c_2\eps <  \delta$.
\end{lem}
\begin{proof} It is sufficient to verify that the homogeneous problem only has the trivial solution. Hence suppose that 
$a_h(u_h,v_h)=0$ for all $v_h \in \Vr$.  Inserting $v_h=u_h$ and using (\ref{normeq1}) we infer 
\begin{displaymath}
\sum_{T \in \Trt} \sum_{i=1}^L \omega_i \, \vr(b_{i,T}) \bigl( | u_h(b_{i,T}) |^2 + | \nabla u_{h|T} |^2 \bigr) \, | T | =0.
\end{displaymath}
The definition of $\Trt$ yields 
\begin{equation}  \label{rholower} 
\vr(b_{i,T}) = \cos^{2(q+1)} \bigl( \frac{\phi(b_{i,T})}{\eps} \bigr) \geq \bigl( \frac{h}{\eps} \bigr)^{2(q+1)} \quad \mbox{ for all } i=1,\ldots,L, \, T \in \Trt,
\end{equation}
so that $ \nabla u_{h|T} \equiv 0, u_h(b_{i,T})=0$ for all $i=1,\ldots,L$ and
all $T \in \Trt$. Hence   $u_h \equiv 0$ in $D_h$.
\end{proof}

Let us  formulate the main result of this paper.

\begin{thm} Let $u \in H^2(\Gamma)$ be the unique solution of (\ref{GP}) extended to $u^e$ via 
(\ref{extension}) and $u_h \in \Vr$ the unique solution of (\ref{WFh}). Let $\eps=\gamma h$ for $\gamma > \frac{1}{r_0}$. Then 
\begin{equation} \label{est1}
\Vert I_h u^e - u_h \Vert_h  \leq  C\, \bigl( h + \gamma^2 h^2 +  \gamma^{-(q+1)}  \bigr) \bigl( \Vert u \Vert_{H^2(\Gamma)} + \Vert f \Vert_{W^{1,\infty}(\Gamma)} \bigr).
\end{equation}
If in addition there exists a constant $c_4>0$ which is independent of $h$ such that
$c_4 h \leq h_T$ for all $ T \in \mathcal T_h$  with $ | T \cap \Gamma |>0$,
then
\begin{equation}  \label{est2}
\Vert u - u_h \Vert_{H^1(\Gamma)}  \leq  C \, \bigl( \sqrt{\gamma} h + \gamma^{\frac{5}{2}} h^2 + \gamma^{-(q+ \frac{1}{2}) }  \bigr) \bigl( \Vert u \Vert_{H^2(\Gamma)} + \Vert f \Vert_{W^{1,\infty}(\Gamma)} \bigr). 
\end{equation}
\end{thm}
The proof of these results will be given in the next section.

\begin{rem} 
The three terms on the right hand side of (\ref{est1}) are related to the different approximations that are used in the discretization. The first term is due
to the use of piecewise linear finite elements in order to discretize the solution and the level set function, while the second term arises from working with
the extended PDE in a narrow band of width $\epsilon = \gamma h$. Here, the factor $\gamma>1$ roughly measures how many grid points are used across the narrow band,
whereas $\gamma^{-(q+1)}$ reflects how well integrals involving the phase field function are approximated via the quadrature rule.
\end{rem}


\section{Error Analysis} \label{sec:3}
\setcounter{equation}{0}

Before we start with the actual error analysis, we first prove a useful  auxiliary result.

\begin{lem}  \label{ahprop}

\begin{equation} \label{phihest}
\eps^{-1} \sum_{T \in \Trt} | \vr |_{0,\infty,T} \Vert v_h \Vert_{1,T}^2 \leq C \Vert v_h \Vert_h^2, \quad \forall \, v_h \in \Vr.
\end{equation}
\end{lem}
\begin{proof} Let us  fix $T \in \Trt$. Using (\ref{vrd}) and Young's inequality we have for every $x \in T$
\begin{eqnarray*}
\vr(x) & \leq & \vr(b_{i,T}) +  | \nabla \vr |_{0,\infty,T} \, | x- b_{i,T}| \leq \vr(b_{i,T}) +  C \, \frac{h_T}{\eps} \, | \vr |_{0,\infty,T}^{\frac{2(q+1)-1}{2(q+1)}} \\
&\leq & \vr(b_{i,T}) +  \frac{1}{2} | \vr |_{0,\infty,T}  + C \,  \bigl( \frac{h}{\eps} \bigr)^{2(q+1)}.
\end{eqnarray*}
Taking the maximum with respect to $x$ and recalling (\ref{rholower}) we infer that
\begin{equation}  \label{rhoinfest}
| \vr |_{0,\infty,T}  \leq C \vr(b_{i,T}), \qquad i=1,\ldots,L.
\end{equation}
To proceed, we choose $x_T \in T$ such that $| v_h(x_T) | 
= | v_h  |_{0,\infty,T}$  and  have,  
as $\nabla v_h$ is constant on $T$,  that
$| v_h(x_T) | \leq | v_h(b_{i,T})| + h_T \,| \nabla v_{h|T} |$ 
for  $i=1,\ldots,L$. Hence, we deduce
\begin{displaymath}
\Vert v_h \Vert_{1,T}^2 \leq \bigl( | v_h(x_T) |^2 + | \nabla v_{h|T} |^2 \bigr) \, | T| \leq C \bigl( | v_h(b_{i,T}) |^2 + | \nabla v_{h|T} |^2 \bigr) \, | T |, \quad i=1,\ldots,L.
\end{displaymath}
Combining this bound with (\ref{rhoinfest}) and observing that $\sum_{i=1}^L \omega_i=1$ we obtain
\begin{displaymath}
| \vr |_{0,\infty,T} \,  \Vert v_h \Vert_{1,T}^2 \leq C \sum_{i=1}^L \omega_i \, \vr(b_{i,T}) \bigl( | v_h(b_{i,T}) |^2 + | \nabla v_{h|T} |^2 \bigr) \, | T | = C \, Q_T \left[ \vr \bigl( | v_h |^2 + | \nabla v_{h|T} |^2 \bigr) \right], 
\end{displaymath}
which concludes the proof of the lemma after summation over $T \in \Trt$. 
\end{proof}

\noindent
Let us now start the proof of the error bound. Define $e_h:= (I_h u^e)_{| D_h} -u_h \in \Vr$. We infer from (\ref{normeq1}) and (\ref{WFh})
\begin{equation} \label{s1s3}
c_3 \Vert e_h \Vert_h^2  \leq  a_h(e_h,e_h) = a_h(I_h u^e,e_h) - l_h(e_h) = :S_1+ S_2.
\end{equation}
Recalling the definition of $a_h$ we may write 
\begin{eqnarray}
\lefteqn{
S_1  =  \eps^{-1} \sum_{T \in \Trt} \left\{ Q_T \left[  \vr \, I_h  \hat{A} \,  \nabla I_h u^e \cdot \nabla e_h  
+ \vr \, I_h  \hat{a}_0  I_h u^e  e_h  \right] \right. } \nonumber  \\ 
& & - \left.  \int_T \left[  \vr \, I_h  \hat{A} \,  \nabla I_h u^e \cdot \nabla e_h  
+ \vr \,  I_h \hat{a}_0  I_h u^e  e_h \right]  \dx \right\}  | \nabla I_h \phi_{|T} |  \nonumber \\
& & + \eps^{-1} \int_{D_h} \bigl( \vr  I_h \hat{A} \,  \nabla (I_h u^e - u^e) \cdot \nabla e_h + \vr  I_h \hat{a}_0 (I_h u^e-u^e) \, e_h \bigr) | \nabla I_h \phi | \dx \nonumber \\ 
& & +  \eps^{-1} \int_{D_h} \bigl( \vr \, (I_h \hat{A} - \hat A)  \nabla u^e  \cdot \nabla e_h + \vr ( I_h \hat{a}_0 - \hat{a}_0) u^e \, e_h \bigr) | \nabla I_h \phi| \dx \nonumber \\ 
& &  + \eps^{-1} \int_{D_h} \vr \, \left[ \hat{A} \, \nabla \ut \cdot \nabla e_h 
+  \hat{a}_0 \,  \ut \, e_h \right]  \, \left[\, | \nabla I_h \phi | 
- | \nabla \phi |\, \right] \dx \nonumber \\
& & + \eps^{-1} \int_{D_h} \vr \, \left[ ( \hat{A} - A^e)\, \nabla u^e \cdot \nabla e_h 
+ (\hat{a}_0 - a_0^e)\, u^e \, e_h \right]  \, | \nabla \phi | \dx \nonumber \\
& & + \eps^{-1} \int_{D_h} \left[ A^e \, \nabla \ut \cdot \nabla e_h+ a_0^e  \, \ut \, e_h \right] \, \vr \, | \nabla \phi | \dx \nn 
 = : \sum_{i=1}^6 S_{1,i}.  \label{S1a}
\end{eqnarray}
Using (\ref{qerror}) and  (\ref{Ihgradbelow}) we obtain 
\begin{eqnarray*}
| S_{1,1} | & \leq & \eps^{-1} \sum_{T \in \Trt} \left\{ | E_T( \vr \, I_h  \hat{A}) | \, | \nabla I_h u^e_{|T} | \, | \nabla e_{h|T} | + | E_T \bigl( \vr \, I_h \hat{a}_0 I_h u^e \, e_h \bigr) | \right\} \nonumber \\
& \leq & C \, \eps^{-1} \sum_{T \in \Trt} | T | h_T^{q+1} \left\{ | \vr \, I_h  \hat{A} |_{q+1,\infty,T} \, | \nabla I_h u^e_{|T} | \, | \nabla e_{h|T} | + | \vr \, I_h  \hat{a}_0 I_h u^e e_h |_{q+1,\infty,T} \right\} \nonumber \\
& \leq & C \, \eps^{-1} \sum_{T \in \Trt} | T | h_T^{q+1} \Vert \vr \Vert_{q+1,\infty,T} \bigl( \Vert I_h \hat{A} \Vert_{1,\infty,T} + \Vert I_h \hat{a}_0 \Vert_{1,\infty,T} \bigr) \Vert I_h u^e \Vert_{1,\infty,T} \Vert e_h \Vert_{1,\infty,T} 
\nonumber \\
& \leq & C \, \eps^{-1} \sum_{T \in \Trt}  h_T^{q+1} \Vert \vr \Vert_{q+1,\infty,T}  \Vert I_h u^e \Vert_{1,T} \Vert e_h \Vert_{1,T} 
\end{eqnarray*}
where the last bound follows from an  inverse estimate and the fact that $a_{ij}, a_0$ are Lipschitz on $\Gamma$. Applying
(\ref{vrd}) and using  (\ref{phihest}), (\ref{inter}), (\ref{inclusions}) and  (\ref{geest}) we deduce
\begin{eqnarray}
| S_{1,1} | & \leq & C \eps^{-1} \sum_{T \in \Trt} h_T^{q+1} \eps^{-(q+1)} | \vr |_{0,\infty,T}^{\frac{1}{2}}  \Vert I_h u^e \Vert_{1,T} \Vert e_h \Vert_{1,T}   \nonumber \\
& \leq & C \, \bigl( \frac{h}{\eps} \bigr)^{q+1} \left( \eps^{-1} \sum_{T \in \Trt} \Vert I_h u^e \Vert_{1,T}^2 \right)^{\frac{1}{2}} \, \left( \eps^{-1} \sum_{T \in \Trt} | \vr |_{0,\infty,T} 
\Vert e_h \Vert_{1,T}^2 \right)^{\frac{1}{2}} \nonumber \\
& \leq & C \, \bigl( \frac{h}{\eps} \bigr)^{q+1} \left( \eps^{-1} \sum_{T \in \Trt} \Vert u^e \Vert_{2,T}^2 \right)^{\frac{1}{2}} \,  \Vert e_h \Vert_h \nonumber \\
& \leq & C \, \bigl( \frac{h}{\eps} \bigr)^{q+1} \eps^{-\frac{1}{2}} \Vert u^e \Vert_{H^2(U_{c_2 \eps})} \Vert e_h \Vert_h \leq 
C \, \bigl( \frac{h}{\eps} \bigr)^{q+1} \Vert u \Vert_{H^2(\Gamma)} \Vert e_h \Vert_h.  \label{S11}
\end{eqnarray}
Using similar arguments we deduce that
\begin{eqnarray}
 |S_{1,2}| &\leq & C \, \eps^{-1} \sum_{T \in \Trt} | \vr |_{0,\infty,T} \Vert I_h u^e - u^e \Vert_{1,T} \Vert e_h \Vert_{1,T} \nonumber \\
 & \leq &  C h  \left( \eps^{-1} \,  \sum_{T \in \Trt}  |  u^e |_{2,T}^2 \right)^{\frac{1}{2}} \, \left( \eps^{-1} \sum_{T \in \Trt} | \vr |_{0,\infty,T} \Vert e_h \Vert_{1,T}^2 \right)^{\frac{1}{2}} \nonumber \\
 & \leq & C h \Vert u \Vert_{H^2(\Gamma)} \Vert e_h \Vert_h \label{S12}
 \end{eqnarray}
as well as 
 \begin{eqnarray}
 |S_{1,3}| + | S_{1,4} | &\leq &  C \,  \eps^{-1} \sum_{T \in \Trt} | \vr |_{0,\infty,T} h_T \bigl( | \hat A |_{1,\infty,T} + | \hat{a}_0 |_{1,\infty,T} + | \phi |_{2,\infty,T} \bigr)  \Vert u^e \Vert_{1,T} \Vert e_h \Vert_{1,T} \nonumber \\
 & \leq &  C h  \left( \eps^{-1} \,  \sum_{T \in \Trt}  \Vert   u^e \Vert_{1,T}^2 \right)^{\frac{1}{2}} \, \left( \eps^{-1} \sum_{T \in \Trt} | \vr |_{0,\infty,T} \Vert e_h \Vert_{1,T}^2 \right)^{\frac{1}{2}} \nonumber \\
 & \leq & C h \Vert u \Vert_{H^1(\Gamma)} \Vert e_h \Vert_h. \label{S13}
 \end{eqnarray}
 Since $A \in C^1(\Gamma)$, it follows from (\ref{pdif}) and (\ref{inclusions}) that for $x \in D_h$ 
\begin{displaymath}
| \hat{A}(x) - A^e(x)| = | A(\hat{p}(x)) - A(p(x))| \leq 
C \,  | \hat{p}(x) - p(x) | \leq  C \,   \phi(x)^2 \leq C \, \epsilon^2
\end{displaymath}
and, similarly, $| \hat{a}_0(x) - a_0^e(x)| \leq C \, \epsilon^2$. This implies together with  (\ref{phihest}) and (\ref{geest}) 
\begin{eqnarray} 
| S_{1,5} | & \leq &      C \,\eps \int_{D_h} \vr \,  \left[\, | \nabla \ut | + 
| \ut |\, \right]\, \left[ \,| \nabla e_h | + | e_h |\, \right]  \dx \nonumber \\
& \leq & C \, \eps^2 \,  \left( \eps^{-1} \,  \sum_{T \in \Trt}  \Vert   u^e \Vert_{1,T}^2 \right)^{\frac{1}{2}} \, \left( \eps^{-1} \sum_{T \in \Trt} | \vr |_{0,\infty,T} \Vert e_h \Vert_{1,T}^2 \right)^{\frac{1}{2}} \nonumber \\
& \leq & C \, \eps^2 \, \eps^{-\frac{1}{2}} \Vert u^e \Vert_{H^1(U_{c_2 \eps})} \, \Vert e_h \Vert_h \leq C\, \eps^2 \,\Vert u \Vert_{H^1(\Gamma)} 
\,\Vert e_h \Vert_h. \label{S15} 
\end{eqnarray}
Combining (\ref{S11})--(\ref{S15}) we infer that
\begin{equation} \label{S1} 
 S_1   \leq   C \left( h + \eps^2 + \bigl( \frac{h}{\eps} \bigr)^{q+1}  \right) \Vert u \Vert_{H^2(\Gamma)} 
\,\Vert e_h \Vert_h +  \eps^{-1} \int_{D_h} \left[ A^e \, \nabla \ut \cdot \nabla e_h+ a_0^e  \, \ut \, e_h \right] \, \vr \, | \nabla \phi | \dx.  
\end{equation}
Next, it follows from  (\ref{defgh}) that
\begin{eqnarray*}
\lefteqn{ 
S_2  =  \eps^{-1} \sum_{T \in \Trt} \left\{  \int_T \vr \, I_h \hat{f} \, e_h \, \dx - Q_T \bigl( \vr \, I_h \hat{f} \, e_h \bigr)  \right\} | \nabla I_h \phi_{|T} | 
+ \eps^{-1} \int_{D_h} \vr (\hat{f} - I_h \hat{f}) e_h \, | \nabla I_h \phi | \dx }  \nonumber \\
& &  + \eps^{-1} \int_{D_h} \vr \, \hat{f} \, e_h  \left[ | \nabla \phi| -  | \nabla I_h \phi |  \right] \dx   + \eps^{-1} \int_{D_h}  \vr  [f^e - \hat{f}  ] \, e_h \, | \nabla \phi | \dx 
   -\eps^{-1}  \int_{D_h} \ft \,   e_h  \, \vr  \, | \nabla \phi | \dx. 
\end{eqnarray*}
Arguing in a similar way as for $S_{1,i}$, $i=1,3,4,5$, we obtain
\begin{equation} \label{S212}
 S_2  \leq  C \left( h + \eps^2 + \bigl( \frac{h}{\eps} \bigr)^{q+1}  \right) \Vert f \Vert_{W^{1,\infty}(\Gamma)} \, \Vert e_h \Vert_h - \eps^{-1}  \int_{D_h} \ft \,   e_h  \, \vr  \, | \nabla \phi | \dx
\end{equation}
and hence 
\begin{equation} \label{s1s2}
S_1 + S_2  \leq  C \left( h +  \epsilon^2 + \bigl( \frac{h}{\eps} \bigr)^{q+1}  \right) \left[ \Vert f \Vert_{W^{1,\infty}(\Gamma)} + \Vert u \Vert_{H^2(\Gamma)} \,\right] \Vert e_h \Vert_h  + S_3,
\end{equation} 
where
\begin{eqnarray} 
S_3 &:= &  \eps^{-1} \int_{D_h} \left[ A^e \, \nabla \ut \cdot \nabla e_h+ a_0^e  \, \ut \, e_h - \ft \, e_h \right] \, \vr \, | \nabla \phi | \dx \label{S3}  \\
& = & \eps^{-1}
\Bigl( \int_{U_{\hat \eps}} + \int_{D_h \setminus U_{\hat \eps}} \Bigr) \left[ A^e \, \nabla \ut \cdot \nabla e_h+ a_0^e  \, \ut \, e_h - \ft \, e_h \right] \, \vr \, | \nabla \phi | \dx = I + II. \nonumber
\end{eqnarray}
If we apply (\ref{WF}) with $r=\hat \eps$ and use the transformation $F$ introduced in Section 2.2 we obtain upon recalling that $\phi(F(p,s))=s$
\begin{eqnarray} 
I & = &  \eps^{-1}  \int_{U_{\hat \eps}} \phi \,R\, e_h \, \vr \, | \nabla \phi |  \dx \label{S23a} \\
& = & \eps^{-1} \int_{-\hat \eps}^{\hat \eps} s \, \sigma\left(\frac{s}{\eps}\right) 
\int_{\Gamma} R(F(p,s)) \, e_h(F(p,s))\,  
| \nabla \phi(F(p,s)) | \, \mu(p,s) \dS_p \ds. \nonumber
\end{eqnarray}
Here, $\mu(p,s)$ is the Jacobian determinant of $F$, which satisfies
\begin{equation}  \label{dif}
\displaystyle  \left| \mu(p,s) - \frac{1}{| \nabla \phi(F(p,s)) | } 
\right| \leq C \,|s | 
\qquad\mbox{for }p \in \Gamma, \;  |s| < \hat \eps.
\end{equation} 
Since
$\displaystyle \int_{-\hat \eps}^{\hat \eps} s
\,\sigma\left(\frac{s}{\eps}\right) 
\ds =0$, we deduce from (\ref{S23a}) that
\begin{align}
I &= \eps^{-1} 
\int_{-\hat \eps}^{\hat \eps} s\,
\sigma\left(\frac{s}{\eps}\right)  \int_{\Gamma}
\left[\, R(F(p,s))\,  | \nabla \phi(F(p,s)) | \, 
\mu(p,s) - R(p) \,\right]\, e_h(F(p,s))\dS_p \ds
\nn \\
& \qquad  + 
\eps^{-1} \int_{-\hat \eps}
^{\hat \eps} s\,
\sigma\left(\frac{s}{\eps}\right)  \int_{\Gamma}
\left[e_h(F(p,s))-e_h(p) \right] R(p) \dS_p \ds
=: I_1 + I_2. 
\label{St32b}
\end{align}
Recalling the form of $R$, (\ref{rest}), as well as $p(F(p,s))=p$ for $p \in \Gamma$, we have
\begin{displaymath}
R(F(p,s)) =  \sum_{1 \leq | \kappa | \leq 2} \bigl(  b_{\kappa}(F(p,s)) + s \, c_{\kappa}(F(p,s)) \bigr) 
    \,D_{\Gamma}^{\kappa} u(p),
\end{displaymath}
so that since $F(p,0)=p$ and  $b_{\kappa} \in C^1$, $c_\kappa \in C^0$ 
\begin{displaymath}
| R(F(p,s))  -  R(p) | \leq C \,|s| \, \sum_{1 \leq | \kappa | \leq 2} | D_{\Gamma}^{\kappa} u(p) | \qquad \forall \, p \in \Gamma, \; |s| < \hat \eps.
\end{displaymath}
Combining this bound with (\ref{dif}) we infer that
\begin{eqnarray}
|I_1| &\leq & C\,\eps^{-1}
\int_{-\hat \eps}^{\hat \eps} s^2\,
\sigma\left(\frac{s}{\eps}\right) 
\int_{\Gamma} \left[\,  | \nabla_{\Gamma} u(p) | + | D_{\Gamma}^2 u(p) | 
\,\right]\,
| e_h(F(p,s))| \dS_p \ds \nn \\
& \leq & C \eps^2 \Vert u \Vert_{H^2(\Gamma)} \left( \eps^{-1} \int_{-\hat \eps}^{\hat \eps} \sigma\left(\frac{s}{\eps}\right)  \int_{\Gamma} | e_h(F(p,s) |^2 \dS_p \ds \right)^{\frac{1}{2}}.
\label{I}
\end{eqnarray}
Similarly, we have that
\begin{eqnarray}
|I_2| &\leq &  
 \eps^{-1}
 \int_{-\hat \eps}^{\hat \eps} | s |\, 
\sigma\left(\frac{s}{\eps}\right)  \int_{\Gamma}
| e_h(F(p,s))- e_h(p)| \, | R(p)| \dS_p \ds \nn \\
& \leq &  C \int_{-\hat \eps}^{\hat \eps}
\sigma\left(\frac{s}{\eps}\right) 
\int_{\Gamma}  | R(p) | \, \left| \int_0^s 
|\nabla e_h(F(p,t))| \dt \right| \dS_p \ds
\nn \\
& \leq &  C\,\eps \int_{-\hat \eps}^{\hat \eps}
\sigma\left(\frac{t}{\eps}\right) 
\int_{\Gamma}  
|\nabla e_h(F(p,t))| \, | R(p) |  \dS_p \dt 
\nn \\  
& \leq &  C\,\eps^2   \Vert u \Vert_{H^2(\Gamma)} \left( \eps^{-1} \int_{-\hat \eps}^{\hat \eps} \sigma\left(\frac{s}{\eps}\right)  \int_{\Gamma} | \nabla e_h(F(p,s) |^2 \dS_p \ds \right)^{\frac{1}{2}},
\label{II}
\end{eqnarray}
where we have  used again (\ref{rest}) as well as  the fact that 
$\sigma \left( \frac{s}{\epsilon} \right) \leq 
\sigma \left( \frac{t}{\epsilon} \right)$ for $|t| \leq |s| 
\leq \hat \eps$.
Combining (\ref{St32b})--(\ref{II}) and applying  once more the transformation rule together with (\ref{inclusions}) and  (\ref{phihest}) we obtain
\begin{equation}
| I | \leq  C \, \epsilon^2 \Vert u \Vert_{H^2(\Gamma)} \left( \eps^{-1} \int_{D_h}
\vr \bigl(  | e_h|^2 +  | \nabla e_h |^2  \bigr)  \dx \right)^{\frac{1}{2}} \leq 
C \,\epsilon^2 \Vert u \Vert_{H^2(\Gamma)}   \,\Vert e_h \Vert_h. \label{S23}
\end{equation}
Since $\vr(x) \leq C \bigl( \frac{h}{\eps} \bigr)^{q+1} \, \sqrt{\vr(x)}, x \in D_h \setminus U_{\hat \eps}$ in view of (\ref{rholow}) we have
\begin{eqnarray*}
| II | & \leq & C \bigl( \frac{h}{\eps} \bigr)^{q+1} \eps^{-1} \int_{D_h \setminus U_{\hat \eps}} \bigl( | \nabla u^e| + | u^e| + | f^e| \bigr)  \, \sqrt{ \vr }  \bigl( | \nabla e_h | +  | e_h |\bigr)    \dx \\
& \leq &  C \bigl( \frac{h}{\eps} \bigr)^{q+1}   \left( \eps^{-1}  \int_{U_{c_2 \eps}} \bigl( | \nabla u^e|^2 +  | u^e|^2 + | f^e |^2 \bigr) \, \dx \right)^{\frac{1}{2}} \,  \left( \eps^{-1} \int_{D_h}
\vr \bigl( | \nabla e_h |^2 + | e_h|^2 \bigr)  \dx \right)^{\frac{1}{2}} \\
& \leq & C \bigl( \frac{h}{\eps} \bigr)^{q+1} \bigl( \Vert u \Vert_{H^1(\Gamma)} + \Vert f \Vert_{L^2(\Gamma)} \bigr) \, \Vert e_h \Vert_h.
\end{eqnarray*}
Inserting the above bounds into (\ref{S3}) we derive
\begin{equation}  \label{S4}
S_3  \leq C \Bigl(\eps^2 + \bigl( \frac{h}{\eps} \bigr)^{q+1} \Bigr) \bigl( \Vert u \Vert_{H^2(\Gamma)} + \Vert f \Vert_{L^2(\Gamma)} \bigr) \Vert e_h \Vert_h,
\end{equation}
so that  (\ref{s1s3}) and (\ref{s1s2}) yield
\begin{displaymath}
\Vert e_h \Vert_h \leq C \Bigl(h +  \eps^2 + \bigl( \frac{h}{\eps} \bigr)^{q+1} \Bigr) \bigl( \Vert u \Vert_{H^2(\Gamma)} + \Vert f \Vert_{W^{1,\infty}(\Gamma)} \bigr),
\end{displaymath}
proving (\ref{est1}). In order to show (\ref{est2}) we shall make use of the following trace--type inequality for $T \in \mathcal T_h$, which is a consequence of \cite[Lemma 3]{HH1} and \cite[Lemma 3]{HH2}:
\begin{equation}  \label{trace}
\Vert v \Vert_{0,T \cap \Gamma}^2 \leq C \bigl(  h_{T}^{-1} \Vert v \Vert_{0,T}^2 +  h_{T} \Vert \nabla v \Vert_{0,T}^2 \bigr) \quad \mbox{ for all } v \in H^1(T).
\end{equation}
If we combine this estimate with (\ref{inter}), the fact that $| \vr |_{0,\infty,T} =1$ if $T \cap \Gamma \neq \emptyset$ and (\ref{inclusions}) we infer that
\begin{eqnarray*}
\lefteqn{
\Vert u - u_h \Vert_{H^1(\Gamma)}^2  =  \sum_{|T \cap \Gamma| >0} \Vert u - u_h \Vert_{H^1(T \cap \Gamma)}^2 = \sum_{|T \cap \Gamma| >0} \Vert u^e - u_h \Vert_{H^1(T \cap \Gamma)}^2 }   \\
& \leq & C \sum_{| T \cap \Gamma | >0} \bigl( h_{T}^{-1} \Vert u^e - u_h \Vert_{1,T}^2 + h_T | u^e |_{2,T}^2 \bigr)  \leq   C \sum_{| T \cap \Gamma | >0}  h_{T}^{-1} | \vr |_{0,\infty,T}  \Vert I_h  u^e - u_h \Vert_{1,T}^2 + C h | u^e |_{2,U_{c_2  \eps}}^2.
\end{eqnarray*}
Finally, using the assumption that $c_4 h \leq h_T$ for all $ T \in \mathcal T_h$  with $ | T \cap \Gamma |>0$, (\ref{phihest}) and (\ref{geest}) we deduce
\begin{displaymath}
  \Vert u - u_h \Vert_{H^1(\Gamma)}^2 
 \leq  C \frac{\eps}{h}   \sum_{| T \cap \Gamma | >0} \eps^{-1}  | \vr |_{0,\infty,T}  \Vert I_h u^e - u_h \Vert_{1,T}^2 + C h \eps \Vert u \Vert_{H^2(\Gamma)}^2   \leq  C \gamma \Vert I_h u^e - u_h \Vert_h^2 + C h \eps \,  \Vert u \Vert_{H^2(\Gamma)}^2,
\end{displaymath}
 from which we infer (\ref{est2}) with the help of (\ref{est1}).


\section{Numerical Experiments} \label{sec:4}
\setcounter{equation}{0}
We investigate the experimental order of convergence (eoc) for the following errors:
$$
\mathcal{E}_1 =\eps^{-1} \sum_{T \in \Trt} Q_T \left[ \vr \, | I_h u^e - u_h|^2 \right]\,~~\mbox{and}~~
\mathcal{E}_2  =  \eps^{-1} \sum_{T \in \Trt} Q_T \left[ \vr \,  | \nabla(I_h u^e - u_h)|^2 \right].
$$
The corresponding calculations will be done for a circle (Example 1) and a sphere (Example 2) of radius 1, described as the zero level set of the function $\phi(x):= |x|^2 -1$. In this case one can verify without difficulty that the projection $p$ constructed in Section \ref{exten} coincides
with the closest point projection $\hat{p}$, so that we have $u^e(x)= u(\frac{x}{|x|})$ for $x \neq 0$. 
We use the finite element toolbox Alberta 2.0, \cite{alberta}, and implement a similar mesh refinement strategy to that in \cite{BNS} with a fine mesh constructed in $D_h$ 
and a coarser mesh in $\Omega \backslash D_h$. The resulting linear systems were solved using CG together with diagonal preconditioning. In all the examples we consider we set $a_{ij}= \delta_{ij}$, $i,j=1,\ldots,n+1$ and $a_0=1$ in (\ref{GP}). 

\subsection*{Example 1}
Let $\Omega = (-1.2, 1.2)^2$ and take  $\Gamma= \lbrace x \in \mathbb{R}^2 \, | \, | x  |= 1 \rbrace$ to be a circle of radius 1,  described as the zero level set of the function $\phi(x):= x_1^2+x_2^2-1$. 
In addition to $\mathcal E_1, \mathcal E_2$ we shall also investigate 
the errors appearing in  (\ref{est2}). To do so, we approximate $\Vert u - u_h \Vert_{L^2(\Gamma)}^2$ and $\Vert \nabla_{\Gamma}(u - u_h) \Vert_{L^2(\Gamma)}^2$ by 
$$
\mathcal{E}_3  = \sum_{l=0}^{L-1}\frac{2\pi}{L}| u(x_l) - u_h(x_l) |^2~~\mbox{and}~~
\mathcal{E}_4  = \sum_{l=0}^{L-1}\frac{2\pi}{L}| \nabla_{\Gamma}u(x_l) - \nabla_{\Gamma} u_h(x_l) |^2
$$
respectively, where we have chosen the  quadrature points
\begin{displaymath}
x_l:= \bigl(\cos(\frac{2 \pi l}{L}), \sin(\frac{2 \pi l}{L}) \bigr)^T, \quad l=0,\ldots,L-1.
\end{displaymath}
In our computations $L=200$ turned out to be sufficient. 
We choose $f$ so that  $u(x):=(x_1^2-x_2^2)/|x|^2$ solves (\ref{GP}) and fixed $\eps=5.333 h$. 
In Table \ref{t:2d1} we display the values of $\mathcal{E}_i$, $i=1,\ldots,4$, together with the eocs, for $q=2$, while in Table \ref{t:2d2} we display $\mathcal{E}_i$, $i=1,\ldots,4$, together with the eocs, for $q=6$. For
the smaller value $q=2$ we observe an eoc for $\mathcal E_2$ which is lower than two indicating
that in this case the term $\gamma^{-(q+1)}$ in (\ref{est1}) dominates. This effect disappears for the choice $q=6$, where we see
eocs close to two for $\mathcal{E}_2$ and $\mathcal{E}_4$. Furthermore, we observe eocs close to four for $\mathcal{E}_1$ and $\mathcal{E}_3$ suggesting that the error analysis can be
improved for the $L^2$--errors.



\begin{table}[!h]
 \begin{center}
 \begin{tabular}{ p{1.6cm}p{1.0cm}p{1.6cm}p{0.7cm}p{1.6cm}p{0.7cm}p{1.6cm}p{0.7cm}p{1.6cm}p{0.6cm}}
 \hline
$h$ & $\varepsilon$ & $\mathcal{E}_1$ & $eoc_1$ & $\mathcal{E}_2$ & $eoc_2$ & $\mathcal{E}_3$ & $eoc_3$ & $\mathcal{E}_4$ & $eoc_4$\\ 
 \hline
 \hline
{3.750e-02} & $0.2$ & {2.150e-05} & - & {1.152e-03} & -& {3.867e-05} & - & {1.555e-02} & -    \\ 
{1.875e-02} & $0.1$ & {1.356e-06} & {3.99} & {2.110e-04} & {2.45} & {2.500e-06} & {3.95} & {3.797e-03} & {2.03}   \\ 
{9.375e-03} & $0.05$ & {7.591e-08} & {4.16} & {9.743e-05} & {1.11} & {1.390e-07} & {4.17} & {9.703e-04} & {1.97}  \\ 
{4.687e-03} & $0.025$ & {4.259e-09} & {4.16} & {9.435e-05} & {0.05} & {7.079e-09} & {4.30} & {2.400e-04} & {2.02}   \\ 
{2.344e-03} & $0.0125$ & {1.806e-10} & {4.56} & {6.677e-05} & {0.50} & {1.721e-10} & {5.36} & {6.007e-05} & {2.00} \\ 
\hline
\end{tabular}
\end{center}
\vspace{-3mm}
\caption{Errors and experimental orders of convergence,  $q=2$}
\label{t:2d1}
\end{table}

\begin{table}[!h]
 \begin{center}
 \begin{tabular}{ p{1.6cm}p{1.0cm}p{1.6cm}p{0.7cm}p{1.6cm}p{0.7cm}p{1.6cm}p{0.7cm}p{1.6cm}p{0.6cm}}
 \hline
$h$ & $\varepsilon$ & $\mathcal{E}_1$ & $eoc_1$ & $\mathcal{E}_2$ & $eoc_2$ & $\mathcal{E}_3$ & $eoc_3$ & $\mathcal{E}_4$ & $eoc_4$\\ 
 \hline
 \hline
{3.750e-02} & $0.2$ & {4.132e-06} & - & {4.552e-04} & -& {1.068e-05} & - & {1.541e-02} & -    \\ 
{1.875e-02} & $0.1$ & {2.570e-07} & {4.01} & {9.600e-05} & {2.25} & {6.707e-07} & {3.99} & {3.739e-03} & {2.04}   \\ 
{9.375e-03} & $0.05$ & {1.603e-08} & {4.00} & {2.293e-05} & {2.07} & {4.194e-08} & {4.00} & {9.527e-04} & {1.97}  \\ 
{4.687e-03} & $0.025$ & {1.005e-09} & {4.00} & {5.701e-06} & {2.01} & {2.631e-09} & {3.99} & {2.357e-04} & {2.02}   \\ 
{2.344e-03} & $0.0125$ & {6.315e-11} & {3.99} & {1.455e-06} & {1.97} & {1.654e-10} & {3.99} & {5.896e-05} & {2.00} \\ \hline
\end{tabular}
\end{center}
\vspace{-3mm}
\caption{Errors and experimental orders of convergence,  $q=6$}
\label{t:2d2}
\end{table}

\subsection*{Example 2}
We set $\Omega = (-1.8, 1.8)^3$ and take   $\Gamma= \lbrace x \in \mathbb{R}^3 \, | \, | x  |= 1 \rbrace$ to be a sphere of radius 1,  described as the zero level set of the function $\phi(x):= x_1^2+x_2^2+x_3^2-1$. 
As in Example 1, 
in addition to $\mathcal E_1, \mathcal E_2$ we shall also investigate 
the errors appearing in  (\ref{est2}) which we approximate by
the quadrature rules
\begin{displaymath}
\mathcal{E}_3  =  \sum_{k=0}^{2L-1} \sum_{l=0}^{L-1} (\frac{\pi}{L})^2 | u(x_{k,l}) - u_h(x_{k,l}) |^2 \, \sin(\frac{l \pi}{L})
\end{displaymath}
and 
\begin{displaymath}
\qquad \quad  \mathcal{E}_4  =  \sum_{k=0}^{2L-1} \sum_{l=0}^{L-1} ( \frac{\pi}{L} )^2 | \nabla_{\Gamma}u(x_{k,l}) - \nabla_{\Gamma} u_h(x_{k,l}) |^2 \, \sin(\frac{l \pi}{L}).
\end{displaymath}
Here,
\begin{displaymath}
x_{k,l} = \bigl( \cos(\frac{k \pi}{L}) \sin( \frac{l \pi}{L}), \sin ( \frac{k \pi}{L}) \sin( \frac{l \pi}{L}), \cos( \frac{l \pi}{L}) \bigr)^T, \quad k=0,\ldots,2L-1, \; l=0,\ldots,L-1
\end{displaymath}
and $L=200$. 
We choose $f$ so that  $u(x):=(x_1^2-x_2^2)/|x|^2$ solves (\ref{GP}) and set $\eps = 5.333 h$.
Due to symmetry, we only solve for $u_h$ over $D_h$ in the positive octant.
In Tables \ref{t:3d1} and \ref{t:3d2} we display the values of $\mathcal{E}_i$, $i=1,\ldots,4$, together with the eocs, for $q=1$ and $q=6$ respectively and observe a similar behaviour as in the two--dimensional test example.   


\subsection*{Example 3}
Here we consider an example similar to the example in Section 9.2 of \cite{DzE13}. We set $\Omega = (-2, 2)^3$ and take $\Gamma$ to be the 
zero level surface of 
\begin{eqnarray*}
\phi(x)& = &(x_1^2-1)^2 +(x_2^2-1)^2 +(x_3^2-1)^2 +(x_1^2+ x_2^2-3)^2 \\
&&+(x_1^2+x_3^2-3)^2+(x_2^2+x_3^2-3)^2- 10.
\end{eqnarray*}
We set 
$f(x)=10000\sin(5(x_1+x_2+x_3)+2.5)$ and take $h=2.2097$e-02, $\varepsilon=0.2$ as well as $q=1$. In Figure \ref{f:3d} we display the approximate solution $u_h$ plotted on the zero level surface of $I_h \phi$.

\begin{table}[!h]
 \begin{center}
 \begin{tabular}{ p{1.6cm}p{0.6cm}p{1.6cm}p{0.7cm}p{1.6cm}p{0.7cm}p{1.6cm}p{0.7cm}p{1.6cm}p{0.6cm}}
 \hline
$h$ & $\varepsilon$ & $\mathcal{E}_1$ & $eoc_1$ & $\mathcal{E}_2$ & $eoc_2$ & $\mathcal{E}_3$ & $eoc_3$ & $\mathcal{E}_4$ & $eoc_4$\\ 
 \hline
 \hline
{7.500e-02} & $0.4$ & {3.425e-05} & - & {5.504e-03} & -& {8.673e-07} & - & {1.978e-03} & -    \\ 
{3.750e-02} & $0.2$ & {6.020e-07} & {5.83} & {5.125e-04} & {3.43} & {1.230e-07} & {2.82} & {4.985e-04} & {1.99}   \\ 
{1.875e-02} & $0.1$ & {1.274e-08} & {5.56} & {8.141e-05} & {2.65} & {9.393e-09} & {3.71} & {9.393e-09} & {3.71}  \\ 
{9.375e-03} & $0.05$ & {3.729e-10} & {5.09} & {2.361e-05} & {1.79} & {5.447e-10} & {4.11} & {3.214e-05} & {2.03}   \\ 
\hline
\end{tabular}
\end{center}
\vspace{-3mm}
\caption{Errors and experimental orders of convergence,  $q=1$}
\label{t:3d1}
\end{table}

\begin{table}[!h]
 \begin{center}
 \begin{tabular}{ p{1.6cm}p{0.6cm}p{1.6cm}p{0.7cm}p{1.6cm}p{0.7cm}p{1.6cm}p{0.7cm}p{1.6cm}p{0.6cm}}
 \hline
$h$ & $\varepsilon$ & $\mathcal{E}_1$ & $eoc_1$ & $\mathcal{E}_2$ & $eoc_2$ & $\mathcal{E}_3$ & $eoc_3$ & $\mathcal{E}_4$ & $eoc_4$\\ 
 \hline
 \hline
{7.500e-02} & $0.4$ & {1.134e-06} & - & {8.248e-04} & -& {1.439e-06} & - & {2.079e-03} & -    \\ 
{3.750e-02} & $0.2$ & {3.627e-08} & {4.97} & {1.440e-04} & {2.52} & {9.382e-08} & {3.94} & {5.034e-04} & {2.05}   \\ 
{1.875e-02} & $0.1$ & {1.721e-09} & {4.40} & {3.212e-05} & {2.16} & {6.245e-09} & {3.91} & {1.308e-04} & {1.94}  \\ 
{9.375e-03} & $0.05$ & {9.899e-11} & {4.12} & {7.789e-06} & {2.04} & {3.820e-10} & {4.03} & {3.197e-05} & {2.03}   \\ 
\hline
\end{tabular}
\end{center}
\vspace{-3mm}
\caption{Errors and experimental orders of convergence,  $q=6$}
\label{t:3d2}
\end{table}

\begin{figure}[h]
\centering
\subfigure{\includegraphics[height = 0.5\textwidth]{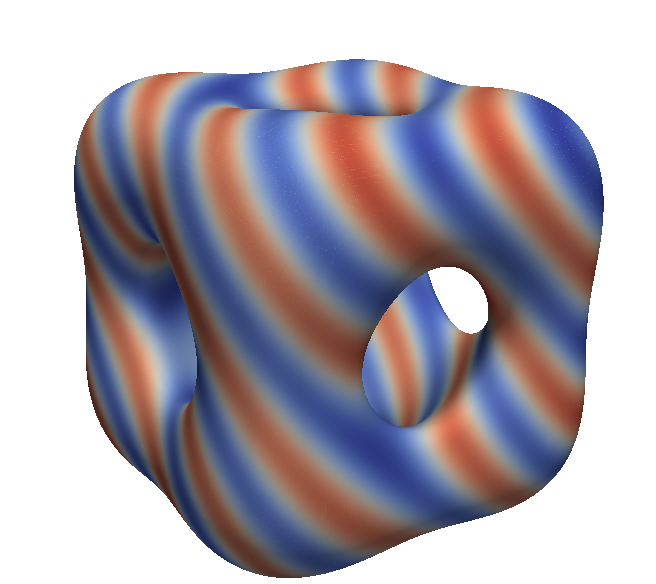}} \hspace{5mm}
\caption{Computational results from Example 3: $u_h$ plotted on the zero level surface of $I_h\phi$. Colouring ranges from the
minimum -86.45 to the maximum 99.57 of the solution.}
\label{f:3d}
\end{figure}

\subsection{Results using piecewise quadratic finite elements}
Even though we have restricted our error analysis to the case of piecewise linear finite elements it is not difficult to apply our approach to quadratic elements. In order to do so, we use
\begin{equation} \label{Vrq}
\Vr := \left\{ v_h \in C(D_h) : v_h \mbox{ is quadratic on } T
\mbox{ for all }  T \in \Trt \right\}
\end{equation}
instead of (\ref{Vr}) and define the   forms $a_h$ and $l_h$ (for the case $a_{ij}=\delta_{ij}, a_0=1$) by
\begin{eqnarray*}
\ah(v_1,v_2) & := &  
\eps^{-1} \sum_{T \in \Trt} Q_T \left[  \vr \, \nabla v_1 \cdot \nabla v_2 \, | \nabla I_h\phi|   
+ \vr \, v_1\,v_2 \, | \nabla I_h \phi | \right] \\
l_h(v) &:= &  
\eps^{-1} \sum_{T \in \Trt} Q_T \left[ \vr\, I_h \hat{f} \,v \, | \nabla I_h\phi | \right],
\end{eqnarray*}
where $I_h$ denotes the Lagrange interpolation operator for piecewise quadratic finite elements. 
The results in Table \ref{t:q1} correspond to the setting outlined in {Example 1}. Using a quadrature rule of order $q=6$ we see eocs close to order four for $\mathcal{E}_2$ and $\mathcal{E}_4$ in contrast to the eocs close to order two, that are displayed in Table \ref{t:2d2}, for the corresponding affine finite element approximation. The fact that the eocs for $\mathcal E_1$ and $\mathcal E_3$ are close to four (rather than six as expected for quadratic elements) is a consequence of
the term $\eps^2=\gamma^2 h^2$ in (\ref{est1}) which now dominates. 

\begin{table}[!h]
 \begin{center}
 \begin{tabular}{ p{1.6cm}p{0.8cm}p{1.6cm}p{0.7cm}p{1.6cm}p{0.7cm}p{1.6cm}p{0.7cm}p{1.6cm}p{0.6cm}}
 \hline
$h$ & $\varepsilon$ & $\mathcal{E}_1$ & $eoc_1$ & $\mathcal{E}_2$ & $eoc_2$ & $\mathcal{E}_3$ & $eoc_3$ & $\mathcal{E}_4$ & $eoc_4$\\ 
 \hline
 \hline
{1.875e-02} & $0.2$ & {2.378e-06} & - & {6.876e-05} & -& {7.063e-06} & - & {3.495e-08} & -    \\ 
{9.375e-03} & $0.1$ & {1.471e-07} & {4.01} & {4.265e-06} & {4.01} & {4.445e-07} & {3.99} & {1.913e-09} & {4.19}   \\ 
{4.687e-03} & $0.05$ & {9.169e-09} & {4.00} & {2.661e-07} & {4.00} & {2.783e-08} & {4.00} & {1.289e-10} & {3.89}  \\ 
{2.344e-03} & $0.025$ & {5.727e-10} & {4.00} & {1.663e-08} & {4.00} & {1.740e-09} & {4.00} & {7.712e-12} & {4.06}   \\ 
{1.172e-03} & $0.0125$ & {3.579e-11} & {4.00} & {1.043e-09} & {3.99} & {1.088e-10} & {4.00} & {4.825e-13} & {4.00} \\ 
\hline
\end{tabular}
\end{center}
\vspace{-3mm}
\caption{Errors and experimental orders of convergence for Example 1, using $(\ref{Vrq})$ with $q=6$}
\label{t:q1}
\end{table}


\noindent {\bf Acknowledgements}\\
VS would like to thank the Isaac Newton Institute for Mathematical Sciences for support and hospitality during the programme {\it Geometry, compatibility and structure preservation in computational differential equations} when work on this paper was undertaken. 
This work was supported by: EPSRC grant number EP/R014604/1.

\begin{appendix}
\setcounter{equation}{0}
\renewcommand{\theequation}{\Alph{section}.\arabic{equation}}
\section{Appendix}

The aim of this appendix is to derive certain properties of the projection $p$ and the extension $u^e(x)=u(p(x))$ which have been used in the
analysis above. To begin, we infer from the definition of $u^e$ for $1 \leq i,j \leq n+1$ that
\begin{eqnarray}
u^e_{x_i}(x) & = &  \sum_{k=1}^{n+1} \underline{D}_k u(p(x))\, p_{k,x_i}(x); \label{uexi} \\
u^e_{x_i x_j}(x) & = &  \sum_{k,l=1}^{n+1} \underline{D}_l\, 
\underline{D}_k u(p(x)) \,p_{k,x_i}(x)\, p_{l,x_j}(x) 
+ \sum_{k=1}^{n+1} \underline{D}_k u(p(x))\, p_{k,x_i x_j}(x). \label{uexixj}
\end{eqnarray}

\begin{lem}
Let $k \in \lbrace 0,1,2 \rbrace$ and $u \in H^k(\Gamma)$. Then 
\begin{equation} \label{geest}
| u^e |_{H^k(U_r)} \leq C \sqrt{r} \Vert u \Vert_{H^k(\Gamma)}, \quad 0<r<\delta.
\end{equation}
\end{lem}
\begin{proof} Using the transformation $F: \Gamma \times (-r,r) 
\rightarrow U_{r}$ with Jacobian determinant $\mu$, (\ref{uexi}), (\ref{uexixj}) and the fact that  $p \in C^2$ we obtain  
\begin{displaymath}
| u^e |_{H^k(U_r)}^2  = \sum_{| \beta | =k}
\int_{U_r} |D^\beta u^e(x)|^2 \dx  \leq C \, \sum_{|\kappa|=0}^{|\beta|} \int_{-r}^{r} \int_{\Gamma}   |D^\kappa_{\Gamma} \,u(p)|^2  \, \mu(p,s)
\dS_p  \ds  \leq   C \,r \,\Vert u \Vert_{H^k(\Gamma)}^2
\end{displaymath}
and the result follows.
\end{proof}

In order to obtain  more precise information about $p$ and its derivatives we  essentially follow the argument in
\cite[Section 2.1]{DER14}, where the corresponding formulae were derived for the case $A=I$. 
For  $x \in U_{\delta}$, we
consider the function
\begin{displaymath}
\eta_x(\tau):= F(p(x),(1-\tau) \phi(x)) = \gamma_{p(x)}((1-\tau) \phi(x)), \quad \tau \in [0,1],
\end{displaymath}
where $\gamma_p$ was defined in (\ref{odesys}). Since $p \in C^2$, it follows that $(x,\tau) \mapsto \eta_x(\tau)$ has continuous partial derivatives of second order
with respect to $x$. 
Clearly, $\eta_x(1)=F(p(x),0)= p(x), \eta_x(0)=F(p(x),\phi(x))=x$. Furthermore, we infer from (\ref{odesys}) that for $k=1,\ldots,n+1$
\begin{equation}  \label{etak}
\eta_{x,k}'(\tau) = - \phi(x) \,\gamma_{p(x),k}'((1- \tau) \phi(x)) 
=  - \phi(x) \, \frac{1}{z_x(\tau)} \sum_{l=1}^{n+1} a_{kl}(p(x)) \phi_{x_l}(\eta_x(\tau)), 
\end{equation}
where $z_x(\tau)=\sum_{r,s=1}^{n+1} a_{rs}(p(x)) \phi_{x_r}(\eta_x(\tau)) \phi_{x_s}(\eta_x(\tau))$. Let us abbreviate $w(x):=z_x(0)$.
The following relations will help to simplify some of the subsequent calculations.

\begin{lem} There exist $d^A_k, d^w \in C^2$ such that 
\begin{eqnarray}
\sum_{l=1}^{n+1} a_{kl}(p(x)) \phi_{x_l}(x) & = & \phi_{x_k}(x) + \phi(x) d^A_k(x), \label{aklnu} \\
w(x) & = &  | \nabla \phi(x) |^2 + \phi(x) d^w(x). \label{z} 
\end{eqnarray}
Furthermore, if $f:\Gamma \rightarrow \mathbb{R}$ is differentiable, then there are $d^f_k \in C^2$ such that
\begin{equation}  \label{dkf}
\sum_{k=1}^{n+1} \underline{D}_k f(p(x)) \phi_{x_k}(x) = \phi(x) \, \sum_{k=1}^{n+1} \underline{D}_k f(p(x)) d_k^f(x). 
\end{equation}
\end{lem}
\begin{proof}
Recalling that $A(p) \nu(p) = \nu(p), \, p \in \Gamma$ as well as $\eta_x(0)=x, \eta_x(1)=p(x)$ we obtain with the help of (\ref{etak})
\begin{eqnarray}
\lefteqn{ \hspace{-1cm} 
\sum_{l=1}^{n+1} a_{kl}(p(x)) \phi_{x_l}(x)  =   \phi_{x_k}(p(x)) +  \sum_{l=1}^{n+1} a_{kl}(p(x)) \bigl( \phi_{x_l}(x)-  \phi_{x_l}(p(x)) \bigr) }   \nonumber \\
& = & \phi_{x_k}(x) -   \sum_{l=1}^{n+1} \bigl( a_{kl}(p(x)) - \delta_{kl} \bigr) \bigl( \phi_{x_l}(\eta_x(1))-  \phi_{x_l}(\eta_x(0)) \bigr)  \nonumber \\
& = &  \phi_{x_k}(x) + \phi(x) d^A_k(x).
\end{eqnarray}
Note that $d^A_k \in C^2$, since this is true for $x \mapsto \eta_x$ and $x \mapsto a_{kl}(p(x))$.  The relation (\ref{z}) immediately follows from (\ref{aklnu}).
Next, observing that $\nabla_{\Gamma} f(p(x)) \in T_{p(x)} \Gamma$ and $\nabla \phi(p(x)) \perp T_{p(x)} \Gamma$ we infer that
\begin{displaymath}
\sum_{k=1}^{n+1} \underline{D}_k f(p(x)) \, \phi_{x_k}(x)  =  \sum_{k=1}^{n+1} \underline{D}_k f(p(x)) \bigl(  \phi_{x_k}(x) - \phi_{x_k}(p(x)) \bigr),   
\end{displaymath}
which implies  (\ref{dkf}) in a similar way as above. 
\end{proof}

Inserting (\ref{aklnu}) and (\ref{z}) into (\ref{etak}) we infer that there exist  $d^{\eta,1}_k \in C^2$ such that 
\begin{equation}  \label{etak1}
\eta_{x,k}'(0)= - \frac{\phi(x) \phi_{x_k}(x)}{| \nabla \phi(x) |^2} + \phi(x)^2 d^{\eta,1}_k(x), \quad 1 \leq k \leq n+1.
\end{equation}
If we differentiate (\ref{etak}) and use again (\ref{etak})  we obtain
\begin{eqnarray}
\lefteqn{
\eta_{x,k}''(\tau) = - \frac{\phi(x)}{z_x(\tau)} \sum_{l,m=1}^{n+1} a_{kl}(p(x)) \phi_{x_l x_m}(\eta_x(\tau)) \eta_{x,m}'(\tau) + \frac{\phi(x)}{z_x(\tau)^2} z_x'(\tau) \sum_{l=1}^{n+1} a_{kl}(p(x)) \phi_{x_l}(\eta_x(\tau)) } \nonumber \\
& = &  \frac{\phi(x)^2}{z_x(\tau)^2} \sum_{l,m,q=1}^{n+1} a_{kl}(p(x)) a_{mq}(p(x)) \phi_{x_l x_m}(\eta_x(\tau)) \phi_{x_q}(\eta_x(\tau)) \label{etak2}  \\
& & -2 \frac{\phi(x)^2}{z_x(\tau)^2} \sum_{l,m,q,r,s=1}^{n+1} a_{kl}(p(x)) a_{mq}(p(x)) a_{rs}(p(x)) \phi_{x_r x_m}(\eta_x(\tau)) \phi_{x_l}(\eta_x(\tau))\phi_{x_q}(\eta_x(\tau)) \phi_{x_s}(\eta_x(\tau)). \nonumber 
\end{eqnarray}
Taylor's theorem together with (\ref{etak1}) and (\ref{etak2}) implies the existence of $d^{p,0}_k \in C^2$ with
\begin{eqnarray} 
p_k(x) & = &  \eta_{x,k}(1) = \eta_{x,k}(0) + \eta_{x,k}'(0) + \int_0^1 (1 - \tau) \eta_{x,k}''(\tau) d \tau \nonumber \\ 
& = &  x_k - \frac{\phi(x) \phi_{x_k}(x)}{| \nabla \phi(x) |^2}  + \phi(x)^2 d^{p,0}_k(x). \label{p1} 
\end{eqnarray}
The relation (\ref{p1}) allows us to prove a bound between $p(x)$ and the closest-point projection $\hat{p}(x)$,
which is used in the error analysis. 

\begin{lem} There exists a constant $C$ such that
\begin{equation}  \label{pdif}
| p(x) - \hat{p}(x) | \leq C\, \phi(x)^2 \qquad \forall \ x \in U_{\delta}.
\end{equation}
\end{lem}
\begin{proof} Let us fix $x \in U_\delta$. Using (\ref{p1}) and the fact that $p(x) \in \Gamma$ we have
\begin{displaymath}
| x - \hat p(x) | \leq | x - p(x) | \leq C | \phi(x) |.
\end{displaymath}
Furthermore, since $T_{\hat p(x)} \Gamma = \mbox{span} \lbrace \nabla \phi(\hat p(x)) \rbrace^\perp$, (\ref{op}) implies that there exists $\lambda \in \mathbb{R}$ such that
$x - \hat{p}(x) = \lambda \, 
\nabla \phi(\hat{p}(x))$. Taylor expansion around $\hat p(x)$ yields together with $\phi(\hat p(x))=0$,  that
\begin{eqnarray*}
\phi(x) & = &  \nabla \phi(\hat p(x)) \cdot (x- \hat p(x)) + \frac{1}{2} (x - \hat p(x))^t  D^2 \phi(\xi) (x - \hat p(x)) \\
& = &  \lambda \, | \nabla \phi(\hat p(x)) |^2 + \frac{1}{2} (x - \hat p(x))^t  D^2 \phi(\xi) (x - \hat p(x))
\end{eqnarray*}
for some $\xi \in [\hat p(x),x]$. Thus
\begin{displaymath}
\lambda = \frac{\phi(x)}{| \nabla \phi(\hat{p}(x))|^2} +  r, \quad \mbox{ where } | r | \leq C | x - \hat p(x) |^2 \leq C \phi(x)^2
\end{displaymath}
and therefore
\begin{equation}  \label{hatp}
 x - \hat{p}(x) =  \phi(x) \,\frac{\nabla \phi(\hat{p}(x))}
 {| \nabla \phi(\hat{p}(x)) |^2} + r \,\nabla \phi(\hat{p}(x)).
\end{equation}
If we combine this relation with (\ref{p1}) we find that
\begin{displaymath}
p(x) - \hat{p}(x) = \phi(x) \left[ \frac{\nabla \phi(\hat{p}(x))}
{| \nabla \phi(\hat{p}(x)) |^2} - \frac{ \nabla \phi(x)}{ | \nabla \phi(x) |^2} \right]
+ r \nabla \phi(\hat p(x)) + \phi(x)^2  \, d^{p,0}(x),
\end{displaymath}
from which we deduce (\ref{pdif}), since $| x - \hat{p}(x)| \leq C \,| \phi(x) |$ and $|r| \leq C \phi(x)^2$.
\end{proof}

Our next aim is to improve on (\ref{p1}) by using a Taylor expansion of one degree higher. We deduce from (\ref{etak2}), (\ref{aklnu}) and (\ref{z}) that
\begin{eqnarray}
\eta_{x,k}''(0) & = &  \frac{\phi(x)^2}{w(x)^2} \sum_{l,m,q=1}^{n+1} a_{kl}(p(x)) a_{mq}(p(x)) \phi_{x_l x_m}(x) \phi_{x_q}(x) \nonumber  \\
& & -2 \frac{\phi(x)^2}{w(x)^2} \sum_{l,m,q,r,s=1}^{n+1} a_{kl}(p(x)) a_{mq}(p(x)) a_{rs}(p(x)) \phi_{x_r x_m}(x) \phi_{x_l}(x)\phi_{x_q}(x) \phi_{x_s}(x) \nonumber \\
& = & \frac{\phi(x)^2}{| \nabla \phi(x) |^4} \sum_{l,m=1}^{n+1} a_{kl}(p(x))  \phi_{x_l x_m}(x) \phi_{x_m}(x) \nonumber   \\
& & -2 \frac{\phi(x)^2 \phi_{x_k}(x)}{| \nabla \phi(x) |^4} \sum_{m,r=1}^{n+1}  \phi_{x_r}(x) \phi_{x_m}(x) \phi_{x_r x_m}(x) + \phi(x)^3 \, d^{\eta,2}_k(x), \label{etak3}
\end{eqnarray}
where $d^{\eta,2}_k \in C^2$. Differentiating (\ref{etak2}) and using (\ref{etak}) as well as (\ref{etak3}) we obtain   
\begin{eqnarray} 
p_k(x) & = &  \eta_{x,k}(1) = \eta_{x,k}(0) + \eta_{x,k}'(0) + \frac{1}{2} \, \eta_{x,k}''(0) + \frac{1}{2}  \int_0^1 (1 - \tau)^2 \eta_{x,k}'''(\tau) d \tau \nonumber \\ 
& = &  x_k - \frac{\phi(x)}{w(x)} \,\sum_{l=1}^{n+1} a_{kl}(p(x)) \,\phi_{x_l}(x)  
+ \frac{1}{2} \frac{\phi(x)^2}{| \nabla \phi(x) |^4} \sum_{l,m=1}^{m+1} a_{kl}(p(x))  \phi_{x_l x_m}(x) \phi_{x_m}(x) \nonumber \\
& & - \frac{\phi(x)^2 \phi_{x_k}(x)}{| \nabla \phi(x) |^4} \sum_{m,r=1}^{n+1} \phi_{x_r}(x) \phi_{x_m}(x) \phi_{x_r x_m}(x) + \phi(x)^3 \, \tilde d^{p,0}_k(x),  \label{pk}
\end{eqnarray}
where $\tilde d^{p,0}_k \in C^2$. \\
Before we continue let us remark
that we may deduce from (\ref{p1}) 
\begin{equation} \label{p2}
p_{k,x_i}(x)  =    \delta_{ik} - \frac{\phi_{x_i}(x) \phi_{x_k}(x)}{| \nabla \phi(x) |^2}  + \phi(x) d^{p,1}_{ik}(x),  \quad 1 \leq i,k \leq n+1,
\end{equation}
where $d^{p,1}_{ik} \in C^1$. Combining this relation with  (\ref{dkf}) we obtain
\begin{eqnarray} \label{akldif}
\lefteqn{ 
\frac{\partial}{\partial x_i} [ a_{kl}(p(x))] = \sum_{m=1}^{n+1} \underline{D}_m a_{kl}(p(x)) p_{m,x_i}(x) } \nonumber \\
& = &  \underline{D}_i a_{kl}(p(x)) - \sum_{m=1}^{n+1} \underline{D}_m a_{kl}(p(x)) \frac{ \phi_{x_i}(x) \phi_{x_m}(x)}{| \nabla \phi(x) |^2} + \phi(x)  \sum_{m=1}^{n+1} \underline{D}_m a_{kl}(p(x)) d^{p,1}_{im}(x) 
\nonumber \\
& = & \underline{D}_i a_{kl}(p(x)) + \phi(x) d^{A,i}_{kl}(x),  \label{diakl}
\end{eqnarray}
where $d^{A,i}_{kl} \in C^1$. 
Differentiating (\ref{pk}) with respect to
$x_i$ and  using (\ref{akldif}), (\ref{z}) we deduce for $1 \leq i,k \leq n+1$
\begin{eqnarray}
p_{k,x_i}(x)  &= &  \delta_{ik} -  \frac{\phi_{x_i}(x)}{w(x)} 
\sum_{l=1}^{n+1} a_{kl}(p(x)) \phi_{x_l}(x) + \phi(x) \frac{\phi_{x_k}(x) w_{x_i}(x)}{| \nabla \phi(x) |^4} - \phi(x) \sum_{l=1}^{n+1} \frac{\underline{D}_i a_{kl}(p(x)) \phi_{x_l}(x)}{| \nabla \phi(x) |^2} \nn  \\
& &  - \phi(x) \sum_{l=1}^{n+1}   \frac{a_{kl}(p(x)) 
\phi_{x_l x_i}(x)}{| \nabla \phi(x) |^2}  + \frac{\phi(x) \phi_{x_i}(x)}{| \nabla \phi(x) |^4} \sum_{l,m=1}^{m+1} a_{kl}(p(x))  \phi_{x_l x_m}(x) \phi_{x_m}(x) \nonumber \\
& & - 2 \frac{\phi(x) \phi_{x_i}(x) \phi_{x_k}(x)}{| \nabla \phi(x) |^4} \sum_{m,r=1}^{n+1} \phi_{x_r}(x) \phi_{x_m}(x) \phi_{x_r x_m}(x) + \phi(x)^2 \tilde d^{p,1}_{ik}(x), \label{pki}
\end{eqnarray}
where $\tilde d^{p,1}_{ik} \in C^1$. 
If we  differentiate this relation with respect to $x_j$ and use (\ref{aklnu}), (\ref{diakl}) we infer for $1 \leq i,j,k \leq n+1$
\begin{eqnarray}
\lefteqn{ \hspace{-1cm}
p_{k,x_i x_j}(x)  =  - \sum_{l=1}^{n+1} \left\{  \frac{\underline{D}_j a_{kl}(p(x)) \phi_{x_i}(x) \phi_{x_l}(x)}{| \nabla \phi(x) |^2} + \frac{ \underline{D}_i a_{kl}(p(x)) \phi_{x_j}(x) \phi_{x_l}(x)}{| \nabla \phi(x) |^2} \right\} }  \nn \\
& & - \sum_{l=1}^{n+1} \left\{  \frac{a_{kl}(p(x)) \phi_{x_i}(x) \phi_{x_l x_j}(x)}{
| \nabla \phi(x) |^2} + \frac{a_{kl}(p(x)) \phi_{x_j}(x) \phi_{x_l x_i}(x)}{
| \nabla \phi(x) |^2} \right\}  \nn \\
& &  + \sum_{l,m=1}^{n+1} \frac{a_{kl}(p(x)) \phi_{x_i}(x) \phi_{x_j}(x) 
\phi_{x_m}(x) \phi_{x_l x_m}(x)}{| \nabla \phi(x) |^4} \nn \\
& & + \beta_{ijk}(x) \phi_{x_k}(x) +
\phi(x) \, \tilde d^{p,2}_{ijk}(x) + \phi(x)^2 \, \tilde d^{p,3}_{ijk}(x), \label{pk2}
\end{eqnarray}
where $\beta_{ijk}, \tilde d^{p,2}_{ij} \in C^1, \tilde d^{p,3}_{ijk} \in C^0$.  Using the above formulae we now obtain:
\begin{lem} \label{extde}
Suppose that $u:\Gamma \rightarrow \mathbb{R}$ is a solution of (\ref{GP}). Then, $u^e$ satisfies (\ref{EP}), (\ref{rest}).
\end{lem}
\begin{proof} Combining (\ref{uexi}), (\ref{p2}) and (\ref{dkf}) we deduce that
\begin{equation} \label{p2a}
u^e_{x_i}(x)  =  \sum_{k=1}^{n+1} \underline{D}_k u(p(x))\, p_{k,x_i}(x)  =  \underline{D}_i u(p(x)) + \phi(x) \sum_{k=1}^{n+1} \alpha^i_k(x) \underline{D}_k u(p(x)),
\end{equation}
where $\alpha^i_k \in C^1$.
Similarly, using (\ref{uexixj}), (\ref{p2}),  (\ref{pk2}) and (\ref{dkf}) we obtain
\begin{eqnarray}
u^e_{x_i x_j}(x) & = &  \sum_{k,l=1}^{n+1} \underline{D}_l \underline{D}_k u(p(x)) p_{k,x_i}(x) p_{l,x_j}(x) + \sum_{k=1}^{n+1} \underline{D}_k u(p(x)) p_{k,x_i x_j}(x) \nonumber  \\
& = & \sum_{k=1}^{n+1} \underline{D}_j \underline{D}_k u(p(x)) \bigl( \delta_{ik} - \frac{\phi_{x_i}(x) \phi_{x_k}(x)}{| \nabla \phi(x) |^2} \bigr) \nonumber \\
& & - \sum_{k,l=1}^{n+1} \underline{D}_k u(p(x))  \left\{  \frac{\underline{D}_j a_{kl}(p(x)) \phi_{x_i}(x) \phi_{x_l}(x)}{| \nabla \phi(x) |^2} + \frac{ \underline{D}_i a_{kl}(p(x)) \phi_{x_j}(x) \phi_{x_l}(x)}{| \nabla \phi(x) |^2} \right\} 
\nonumber \\
& & - \sum_{k,l=1}^{n+1} \underline{D}_k u(p(x)) \left\{  \frac{a_{kl}(p(x)) \phi_{x_i}(x) \phi_{x_l x_j}(x)}{
| \nabla \phi(x) |^2} + \frac{a_{kl}(p(x)) \phi_{x_j}(x) \phi_{x_l x_i}(x)}{| \nabla \phi(x) |^2} \right\} \nonumber \\
& & + \sum_{k,l,m=1}^{n+1} \underline{D}_k u(p(x))  \frac{a_{kl}(p(x)) \phi_{x_i}(x) \phi_{x_j}(x) 
\phi_{x_m}(x) \phi_{x_l x_m}(x)}{| \nabla \phi(x) |^4} \nonumber \\
& & + \sum_{1 \leq | \kappa | \leq 2} \bigl( \phi(x) \alpha^{ij}_{\kappa}(x) + \phi(x)^2 \tilde \alpha^{ij}_{\kappa}(x) \bigr) D_{\Gamma}^{\kappa} u(p(x)), \label{p2b}
\end{eqnarray}
where $\alpha^{ij}_\kappa \in C^1, \tilde \alpha^{ij}_\kappa \in C^0$.
Recalling (\ref{aklnu}) and using (\ref{p2b}) and  the symmetry of the coefficients $a_{ij}$ we infer that
\begin{eqnarray}
\lefteqn{
\sum_{i,j=1}^{n+1} a_{ij}^e(x) \, u^e_{x_i x_j}(x) = \sum_{i,j=1}^{n+1} a_{ij}(p(x)) \, u^e_{x_i x_j}(x) } \nonumber \\
& = & \sum_{i,j=1}^{n+1} a_{ij}(p(x)) \underline{D}_j \underline{D}_i u(p(x)) - \sum_{j,k=1}^{n+1} \underline{D}_j \underline{D}_k u(p(x)) \, \frac{\phi_{x_j}(x) \phi_{x_k}(x)}{| \nabla \phi(x) |^2} \nonumber  \\
& & - 2 \sum_{j,k,l=1}^{n+1} \underline{D}_k u(p(x)) \underline{D}_j a_{kl}(p(x)) \frac{\phi_{x_j}(x) \phi_{x_l}(x)}{ | \nabla \phi(x) |^2} 
-\sum_{k,l,m=1}^{n+1} a_{kl}(p(x)) \underline{D}_k u(p(x)) \frac{\phi_{x_m}(x) \phi_{x_l x_m}(x)}{ | \nabla \phi(x) |^2} \nonumber \\
& & + \sum_{1 \leq | \kappa | \leq 2} \bigl( \phi(x) \beta_{\kappa}(x) + \phi(x)^2 \tilde \beta_{\kappa}(x) \bigr) D_{\Gamma}^{\kappa} u(p(x)) \nonumber \\
& = & \sum_{i,j=1}^{n+1} a_{ij}(p(x)) \underline{D}_j \underline{D}_i u(p(x))
-\sum_{k,l,m=1}^{n+1} a_{kl}(p(x)) \underline{D}_k u(p(x)) \frac{\phi_{x_m}(x) \phi_{x_l x_m}(x)}{ | \nabla \phi(x) |^2} \nonumber \\
& & + \sum_{1 \leq | \kappa | \leq 2} \bigl( \phi(x) \gamma_{\kappa}(x) + \phi(x)^2 \tilde \gamma_{\kappa}(x) \bigr) D_{\Gamma}^{\kappa} u(p(x)),
\label{aij1}
\end{eqnarray}
where  the last identity follows from (\ref{dkf}) and where $\gamma_\kappa \in C^1, \tilde \gamma_\kappa \in C^0$.
On the other hand, (\ref{akldif}) and (\ref{p2a}) yield 
\begin{equation} \label{aijxj}
\sum_{i,j=1}^{n+1} a^e_{ij,x_j}(x) u^e_{x_i}(x)= \sum_{i,j=1}^{n+1} \underline{D}_j a_{ij}(p(x)) \, \underline{D}_i u(p(x)) + \phi(x) \sum_{k=1}^{n+1} \tilde \beta_k(x) \underline{D}_k u(p(x)),
\end{equation}
where $\tilde \beta_k \in C^1$.
Combining (\ref{aij1}) and (\ref{aijxj}) we find that
\begin{eqnarray*}
\lefteqn{
\frac{1}{| \nabla \phi(x) |} \nabla \cdot \left( A^e(x) \nabla u^e(x) | \nabla \phi(x)| \right) }  \\
&   = &   \sum_{i,j=1}^{n+1} \bigl( a^e_{ij,x_j}(x) u^e_{x_i}(x) + a^e_{ij}(x) 
u^e_{x_i x_j}(x) \bigr)  + \frac{1}{| \nabla \phi(x) |^2} 
\sum_{i,j,k=1}^{n+1}  a^e_{ij}(x) u^e_{x_i}(x) \phi_{x_k}(x) \phi_{x_k x_j}(x) \\
& = &  \sum_{i,j=1}^{n+1} \underline{D}_j \bigl( a_{ij}(p(x)) 
\underline{D}_i u(p(x)) \bigr) +  \sum_{1 \leq | \kappa | \leq 2} \bigl( \phi(x) b_{\kappa}(x) + \phi(x)^2 c_\kappa(x) \bigr) D_{\Gamma}^{\kappa} u(p(x)),
\end{eqnarray*}
where $b_\kappa \in C^1, c_\kappa \in C^0$. Combining this relation  with (\ref{GP}) implies (\ref{EP}) and  (\ref{rest}). 
\end{proof}


\end{appendix}

\bibliographystyle{plain}


\end{document}